\documentclass[12pt,a4paper,reqno]{amsart}
\usepackage[a4paper,left=16mm,right=16mm,top=36mm,bottom=36mm]{geometry}

\newtheorem{dfntn}{Definition}[section] 
\newtheorem{rmrk}{Remark}[section]    
\newtheorem{thrm}{Theorem}[section]     
\newtheorem{lmm}{Lemma}[section]	

\usepackage{amssymb}
\usepackage{amsmath}
\usepackage{graphicx}
\usepackage{float}
\usepackage[utf8]{inputenc}
\usepackage[english]{babel}
\usepackage{subfig}
\usepackage{stmaryrd}
\newcommand{\R}{\mathbb{R}}
\newcommand{\N}{\mathbb{N}}
\newcommand{\eps}{\varepsilon}

\newcommand{\vertiii}[1]{{\left\vert\kern-0.25ex\left\vert\kern-0.25ex\left\vert #1
    \right\vert\kern-0.25ex\right\vert\kern-0.25ex\right\vert}}

\def\calO{\mathcal{O}}

\def\calT{\mathcal{T}}
\def\calS{\mathcal{S}}

\numberwithin{equation}{section}

\usepackage{tikz}
\usepackage{pgfplots}
\pgfplotsset{compat=1.15}

\usepackage{hyperref}

\begin{document}

\title[Mixed FE approximation of periodic HJB problems and numerical homogenization]{Mixed finite element approximation of periodic Hamilton--Jacobi--Bellman problems with application to numerical homogenization}

\author[D. Gallistl]{Dietmar Gallistl}
\address[Dietmar Gallistl]{Friedrich-Schiller-Universit\"{a}t Jena, Institut f\"ur Mathematik, Ernst-Abbe-Platz 2, 07743 Jena, Germany.}
\email{dietmar.gallistl@uni-jena.de}

\author[T. Sprekeler]{Timo Sprekeler}
\address[Timo Sprekeler]{University of Oxford, Mathematical Institute, Woodstock Road, Oxford OX2 6GG, UK.}
\email{sprekeler@maths.ox.ac.uk}

\author[E. S\"{u}li]{Endre S\"{u}li}
\address[Endre S\"{u}li]{University of Oxford, Mathematical Institute, Woodstock Road, Oxford OX2 6GG, UK.}
\email{suli@maths.ox.ac.uk}

\subjclass[2010]{35B27, 35J60, 65N12, 65N15, 65N30}
\keywords{Hamilton--Jacobi--Bellman equation, nondivergence-form elliptic PDE, Cordes condition, mixed finite element methods, homogenization}
\date{\today}

\begin{abstract}
In the first part of the paper, we propose and rigorously analyze a mixed finite element method for the approximation of the periodic strong solution to the fully nonlinear second-order Hamilton--Jacobi--Bellman equation with coefficients satisfying the Cordes condition. These problems arise as the corrector problems in the homogenization of Hamilton--Jacobi--Bellman equations. The second part of the paper focuses on the numerical homogenization of such equations, more precisely on the numerical approximation of the effective Hamiltonian. Numerical experiments demonstrate the approximation scheme for the effective Hamiltonian and the numerical solution of the homogenized problem.
\end{abstract}

\maketitle

\section{Introduction}

In the first part of this work we consider the numerical approximation of the periodic boundary-value problem for the elliptic Hamilton--Jacobi--Bellman (HJB) equation
\begin{align}\label{intro HJB equ}
\begin{aligned}\sup_{\alpha\in\Lambda}\left\{  L^{\alpha}u - f^{\alpha}\right\} = 0 \quad \text{in }Y,\qquad
u\text{ is $Y$-periodic},\end{aligned}
\end{align}
where $\Lambda$ denotes a compact metric space and $Y:=(0,1)^n\subset\R^n$ is the unit cell in dimension $n\geq 2$. Here, $\{L^{\alpha}\}_{\alpha\in\Lambda}$ denotes the parametrized family of the linear uniformly elliptic (see \eqref{unifm ell}) differential operators
\begin{align*}
L^{\alpha}w:=-A^{\alpha}:D^2 w - b^{\alpha}\cdot \nabla w + c^{\alpha} w  :=-\sum_{i,j=1}^n a_{ij}(\,\cdot\,,\alpha)\,\partial_{ij}^2 w - \sum_{i=1}^n b_{i}(\,\cdot\,,\alpha)\,\partial_{i} w+ c(\,\cdot\,,\alpha)\, w
\end{align*}
and $f^{\alpha}:=f(\,\cdot\,,\alpha)$ with uniformly continuous functions $a_{ij}=a_{ji},b_i,c,f\in C(\R^n\times\Lambda)$ and positive zeroth-order coefficient $c>0$.

It is assumed that $A^{\alpha},b^{\alpha},c^{\alpha},f^{\alpha}$ are $Y$-periodic on $\R^n$ and that the coefficients satisfy the Cordes condition, i.e., that there exist constants $\lambda>0$ and $\delta\in (0,1)$ such that
\begin{align}\label{intro cordes}
\lvert A^{\alpha}\rvert^2+\frac{\lvert b^{\alpha}\rvert^2}{2\lambda} + \frac{(c^{\alpha})^2}{\lambda^2}\leq \frac{1}{n+\delta}\left(\mathrm{tr}(A^{\alpha}) + \frac{c^{\alpha}}{\lambda}  \right)^{2}
\end{align}
holds in $\R^n$ for all $\alpha\in\Lambda$. Under these assumptions, the periodic HJB problem \eqref{intro HJB equ} admits a unique strong solution $u\in H^2_{\mathrm{per}}(Y)$; see Section \ref{Section 2 framework}.

Problems of the form \eqref{intro HJB equ} arise naturally in the homogenization of HJB equations, which is the focus of the second part of this work. We are concerned with elliptic homogenization problems of the form
\begin{align}\label{intro u_eps equ}
\left\{ \begin{aligned} u_{\eps}+F\left[x,\frac{x}{\eps},\nabla u_{\eps},D^2 u_{\eps}\right] &= 0 &\quad &\text{in }\Omega,\\ \hfill u_{\eps}&= 0 &\quad  &\text{on }\partial\Omega,\end{aligned}\right.
\end{align}
with $\Omega\subset\R^n$ being a convex domain in dimension $n\geq 2$, a small parameter $\eps>0$, and
\begin{align*}
F\left[x,y,\nabla w,D^2 w\right]:=\sup_{\alpha\in\Lambda}\left\{-A^{\alpha}(x,y):D^2 w - b^{\alpha}(x,y)\cdot \nabla w  -f^{\alpha}(x,y)\right\}
\end{align*}
with uniformly continuous coefficients $a_{ij}=a_{ji},b_i,f\in C(\bar{\Omega}\times\R^n\times\Lambda)$. It is assumed that $A^{\alpha},b^{\alpha},f^{\alpha}$ are $Y$-periodic in $y\in\R^n$ with respect to their second arguments, and that the coefficients satisfy the Cordes condition \eqref{intro cordes} and the Lipschitz condition on $\bar{\Omega}\times \R^n$ uniformly in $\alpha\in\Lambda$.

It is well-known (see e.g., Caffarelli, Souganidis, Wang \cite{CSW05}, Evans \cite{Eva89,Eva92}) that the viscosity solution $u_{\eps}\in C(\bar{\Omega})$ to \eqref{intro u_eps equ} converges uniformly, as $\eps\searrow 0$, to the viscosity solution $u_0\in C(\bar{\Omega})$ of the homogenized problem
\begin{align}\label{intro hom}
\left\{ \begin{aligned} u_0 + H(x,\nabla u_0, D^2 u_0) &= 0 &\quad &\text{in }\Omega,\\ \hfill u_{0}&= 0 &\quad  &\text{on }\partial\Omega,\end{aligned}\right.
\end{align}
for some function $H:\bar{\Omega}\times \R^n\times \calS^{n\times n}\rightarrow \R$, the so-called effective Hamiltonian (here $\calS^{n\times n}:=\R^{n\times n}_{\mathrm{sym}}$). The value of the effective Hamiltonian at a fixed point $(s,p,R)\in \bar{\Omega}\times \R^n\times \calS^{n\times n}$ can be obtained as the uniform limit of the sequence $\{-\sigma v^{\sigma}\}_{\sigma>0}$ as $\sigma\searrow 0$, where the so-called approximate corrector $v^{\sigma}=v^{\sigma}(\cdot\,;s,p,R)$ is the solution to the problem
\begin{align}\label{intro appr cor}
\sigma v^{\sigma}+F\left[s,y,p,R+D_y^2 v^{\sigma}\right] = 0\quad \text{for }y\in Y,\qquad y\mapsto v^{\sigma}(y;s,p,R) \text{ is $Y$-periodic};
\end{align}
see e.g., Alvarez, Bardi \cite{AB01,AB10}, Camilli, Marchi \cite{CM09}. Observe that the problem for the approximate corrector \eqref{intro appr cor} fits into the framework of \eqref{intro HJB equ}. For further homogenization results we refer to Section \ref{Section NumHom}.

The main goal of this work is the efficient numerical approximation of the effective Hamiltonian. In order to do so, we first propose and analyze the numerical approximation of periodic boundary-value problems of the type \eqref{intro HJB equ} by a mixed finite element scheme (Section \ref{Section FE schemes for periodic HJB}), and then proceed with the numerical study of the approximate correctors and the effective Hamiltonian (Section \ref{Section NumHom}).

The motivation for studying the fully nonlinear second-order Hamilton--Jacobi--Bellman equation comes from stochastic control theory for Markov diffusion processes and we refer the reader to Fleming, Soner \cite{FS06}. Its study is a mathematically challenging task as there is no natural variational formulation and solvability has to be considered either in the sense of viscosity solutions (see Definition \ref{dfntn visc} and the user's guide by Crandall, Ishii, Lions \cite{CIL92} for a comprehensive overview), or in the sense of strong solutions, i.e., functions admitting weak derivatives up to order two satisfying the equation pointwise almost everywhere.

The finite element approximation of periodic HJB problems has not been studied a lot so far; the finite element approximation of the Dirichlet problem, however, has been the focus of active research over the past decade; see Feng, Glowinski, Neilan \cite{FGN13} and Neilan, Salgado, Zhang \cite{NSZ17} for a survey of recent developments. The mixed finite element method presented in this paper is a modified version of the mixed scheme for the Dirichlet problem with Cordes coefficients introduced in the previous work by Gallistl, S\"{u}li \cite{GS19}, which enables the use of $H^1$-conforming finite elements. For further $H^1$-conforming finite element schemes, we refer to Camilli, Falcone \cite{CF95}, Camilli, Jakobsen \cite{CJ09}, Jensen \cite{Jen17}, and Jensen, Smears \cite{JS13}. The first numerical scheme for HJB equations in the Cordes framework has been the discontinuous Galerkin finite element method in Smears, S\"{u}li \cite{SS14,SS16}.

The numerical homogenization of second-order HJB problems, and nondivergence-form problems in general, has not been studied extensively so far. For the case of linear nondivergence-form equations, we refer the reader to the previous work Capdeboscq, Sprekeler, S\"{u}li \cite{CSS20} (see also Sprekeler, Tran \cite{ST20}), and to the references therein. For the case of second-order HJB equations, a finite difference scheme for the whole-space problem has been proposed in Camilli, Marchi \cite{CM09}. In Finlay, Oberman \cite{FO18,FOO18}, the effective Hamiltonian is computed exactly for HJB operators of certain types and numerical simulations have been conducted. It seems that finite element schemes for the numerical homogenization of the problem \eqref{intro u_eps equ} have not yet been constructed. 

Let us note that there is a lot more work in the literature on the numerical approximation of the effective Hamiltonian arising in the homogenization of first-order Hamilton--Jacobi equations; see various authors \cite{ACC08,FR08,GLQ18,GO04,LYZ11,OTV09,Qia03,QTY18}. The paper is structured as follows.

In Section \ref{Section FE schemes for periodic HJB}, we propose and rigorously analyze a mixed finite element method for the approximation of the periodic solution to the HJB equation \eqref{intro HJB equ}. We prove \textit{a priori} (see Theorem \ref{Thm error bd}) as well as \textit{a posteriori} (see Theorem \ref{thrm a posteriori}, Remark \ref{rmrk local efficiency}) error bounds with explicit error constants.

In Section \ref{Section NumHom}, we discuss the numerical homogenization of problems of the form \eqref{intro u_eps equ}. We provide the framework and theoretical homogenization results in Sections \ref{subsec fra} and \ref{subsec homogenization} respectively. We then analyze the approximation of the approximate corrector \eqref{intro appr cor} by the mixed finite element scheme from Section \ref{Section FE schemes for periodic HJB}, and present a scheme for the approximation of the effective Hamiltonian in Sections \ref{subsec mixed fem for appr corr} and \ref{subsec app of eff ham} respectively.

In Section \ref{Sec Num Exp}, we present numerical experiments for the approximate corrector problems and the homogenized effective equation.

In Section \ref{Sec Coll of Pfs}, we collect the proofs of the results contained in this work. Let us note that some proofs follow certain arguments of the earlier work \cite{GS19}. Here, it is important to track the dependence of error constants in the Cordes parameters (see Remark \ref{rk explicit Ce}), which is crucial for the arguments in Sections \ref{subsec mixed fem for appr corr} and \ref{subsec app of eff ham}. We have therefore included all details of the proofs.

For simplicity, all results in this work are presented for dimensions $n\in \{2,3\}$ in which we define the rotation (curl) of a sufficiently regular vector field $w=(w_i)_{1\leq i\leq n}:\R^n\rightarrow\R^n$ by
\begin{align*}
\mathrm{rot}(w):=\partial_2 w_1 - \partial_1 w_2\quad\text{if }n=2,\quad\;\;
\mathrm{rot}(w):=(\partial_2 w_3 - \partial_3 w_2, \partial_3 w_1 - \partial_1 w_3,\partial_1 w_2 - \partial_2 w_1)\quad\text{if }n=3.
\end{align*}
The results in this paper remain valid for higher dimensions $n\geq 4$, in which case the rotation operator needs to be replaced by the exterior derivative operator.

\section{Mixed FEM for Periodic HJB Problems}\label{Section FE schemes for periodic HJB}

\subsection{Framework}\label{Section 2 framework}
In dimension $n\in \{2,3\}$, we let $Y:=(0,1)^n$ denote the unit cell in $\R^n$. Further, we let $\Lambda$ be a compact metric space. We then consider the problem of finding periodic strong solutions to the second-order Hamilton--Jacobi--Bellman equation
\begin{align}\label{HJB equ}
\sup_{\alpha\in\Lambda} \left\{ -A^{\alpha}:D^2 u - b^{\alpha}\cdot \nabla u + c^{\alpha} u -f^{\alpha}\right\} = 0 \quad &\text{in }Y,\qquad u \in H^2_{\mathrm{per}}(Y),
\end{align}
where we make the following assumptions on the coefficients: Writing $\calS^{n\times n}\subset \R^{n\times n}$ for the space of real symmetric $n\times n$ matrices, we assume that the functions
\begin{align*}
A=(a_{ij})_{1\leq i,j\leq n}:\R^n\times \Lambda&\rightarrow \calS^{n\times n},& (y,\alpha)&\mapsto A(y,\alpha)=:A^{\alpha}(y),\\
b=(b_i)_{1\leq i\leq n}:\R^n\times\Lambda&\rightarrow \R^n,& (y,\alpha)&\mapsto b(y,\alpha)=:b^{\alpha}(y),\\
c:\R^n\times\Lambda&\rightarrow \R,& (y,\alpha)&\mapsto c(y,\alpha)=:c^{\alpha}(y),\\
f:\R^n\times\Lambda&\rightarrow \R,& (y,\alpha)&\mapsto f(y,\alpha)=:f^{\alpha}(y)
\end{align*}
are $Y$-periodic in $y\in\R^n$ and uniformly continuous, i.e., $a_{ij},b_i,c,f\in C(\R^n\times \Lambda)$. Further, we assume that the zeroth-order coefficient is positive,
\begin{align*}
\inf_{\R^n\times \Lambda} c >0,
\end{align*}
that the matrix-valued function $A$ is uniformly elliptic,
\begin{align}\label{unifm ell}
\exists\,\zeta_1,\zeta_2>0:\quad \zeta_1 \lvert \xi\rvert^2\leq  A(y,\alpha)\xi \cdot \xi\leq \zeta_2 \lvert \xi\rvert^2\quad \forall y,\xi\in \R^n,\, \alpha\in \Lambda,
\end{align}
and that the Cordes condition
\begin{align}\label{Cordes periodic}
\lvert A\rvert^2+\frac{\lvert b\rvert^2}{2\lambda} + \frac{c^2}{\lambda^2}\leq \frac{1}{n+\delta}\left(\mathrm{tr}( A) + \frac{c}{\lambda}  \right)^{2}   \qquad\text{in }\R^n\times \Lambda
\end{align}
holds for some constants $\lambda>0$ and $\delta\in (0,1)$. We refer to \cite{SS14} for a brief discussion of the Cordes condition.

Let us introduce the function $\gamma\in C(\R^n\times\Lambda)$ defined by
\begin{align}\label{gamma first def}
\gamma:= \left(\lvert A\rvert^2+\frac{\lvert b\rvert^2}{2\lambda} + \frac{c^2}{\lambda^2}  \right)^{-1} \left(\mathrm{tr}( A) + \frac{c}{\lambda}\right).
\end{align}
We let $\gamma^{\alpha}:\R^n\rightarrow\R$, $\gamma^{\alpha}(y):=\gamma(y,\alpha)$ for $\alpha\in\Lambda$, and consider the renormalized Hamilton--Jacobi--Bellman equation
\begin{align}\label{HJB modified}
\sup_{\alpha\in\Lambda} \left\{\gamma^{\alpha}\left( -A^{\alpha}:D^2 u - b^{\alpha}\cdot \nabla u + c^{\alpha} u -f^{\alpha}\right)\right\} = 0 \quad &\text{in }Y,\qquad u \in H^2_{\mathrm{per}}(Y).
\end{align}
The function $\gamma$ takes the role of a (positive) multiplying factor for the equation; note that $\inf_{\R^n\times \Lambda}\gamma >0$. Hence $u\in H^2_{\mathrm{per}}(Y)$ is a solution of \eqref{HJB equ} if, and only if, it is a solution of \eqref{HJB modified}.

It can then be shown that the problem \eqref{HJB equ} is well-posed in the sense of strong solutions, following the steps of the proof of \cite[Theorem 3]{SS14} (note the sign-difference in the coefficient functions). The proof is omitted.

\begin{thrm}[Well-posedness]\label{thm well-pos}
In the situation described above, there exists a unique strong solution $u\in H^2_{\mathrm{per}}(Y)$ to the problem \eqref{HJB equ}. Further, $u$ is also the unique strong solution to the problem \eqref{HJB modified}.
\end{thrm}

\subsection{Mixed formulation of the problem}\label{subsec mixed fem}

We construct a mixed finite element method for the numerical approximation of the strong periodic solution to \eqref{HJB equ} similarly to the scheme presented in \cite{GS19}. The mixed formulation relies on rewriting the problem \eqref{HJB modified} as
\begin{align}\label{mixed idea}
\sup_{\alpha\in\Lambda} \left\{\gamma^{\alpha}\left(-A^{\alpha}:D w -b^{\alpha}\cdot \nabla u  + c^{\alpha} u  -f^{\alpha}\right)\right\} = 0,\qquad w = \nabla u.
\end{align}
We denote the space of functions $v\in H^1_{\mathrm{per}}(Y)$ with zero mean over the unit cell $Y$ by
\begin{align*}
W_{\mathrm{per}}(Y):= \left\{v\in H^1_{\mathrm{per}}(Y):\; \int_Y v = 0\right\}.
\end{align*}
We further let $W_{\mathrm{per}}(Y;\R^n):= \left(W_{\mathrm{per}}(Y)\right)^n$ and denote the Jacobian of a function $w\in W_{\mathrm{per}}(Y;\R^n)$ by $Dw$. Noting that $Y=(0,1)^n$ is convex and $\mathrm{diam}(Y)=\sqrt{n}$, we have the Poincar\'e inequality (see \cite[Theorem 3.2]{Beb03}) for scalar functions,
\begin{align}\label{Poincare explicit scalar}
\|v\|_{L^2(Y)}\leq \frac{\sqrt{n}}{\pi}\|\nabla v\|_{L^2(Y)} \qquad\forall\, v\in W_{\mathrm{per}}(Y),
\end{align}
and the corresponding inequality for vector-valued functions,
\begin{align}\label{Poincare explicit}
\|w\|_{L^2(Y)}\leq \frac{\sqrt{n}}{\pi}\|D w\|_{L^2(Y)} \qquad\forall\, w\in W_{\mathrm{per}}(Y;\R^n).
\end{align}
Noting that a solution $u\in H^2_{\mathrm{per}}(Y)$ to \eqref{mixed idea} satisfies $w=\nabla u\in W_{\mathrm{per}}(Y;\R^n)$, we define the function space
\begin{align*}
X:= W_{\mathrm{per}}(Y;\R^n)\times H^1_{\mathrm{per}}(Y).
\end{align*}
Further, we let $M\subset W_{\mathrm{per}}(Y)$ be a closed linear subspace. Admissible choices include $M=\{0\}$ and $M=W_{\mathrm{per}}(Y)$.

The mixed formulation is defined as the following problem: Find $m\in M$ and $(w,u)\in X$ such that
\begin{align}\label{Mixed formulation}
\left\{ \begin{aligned} a\left((w,u),(w',u') \right)+ b(m,(w',u')) &= 0 &\quad &\forall \,(w',u')\in X,\\ b(m',(w,u))&=0&\quad & \forall\, m'\in M,\end{aligned}\right.
\end{align}
where the semilinear form $a:X\times X\rightarrow \R$ is given by
\begin{align*}
a\left((w,u),(w',u') \right):= \int_Y F_{\gamma}[(w,u)]\,L_{\lambda}(w',u')  + \sigma_1 \int_Y \mathrm{rot}(w)\cdot\mathrm{rot}(w')+\sigma_2 \int_Y \left(\nabla u - w\right)\cdot \left(\nabla u'-w'\right),
\end{align*}
and the bilinear form $b:M\times X\rightarrow \R$ is given by
\begin{align*}
b(m,(w,u)):=\int_Y \nabla m\cdot \left(\nabla u-w\right)
\end{align*}
for $(w,u),(w',u')\in X$ and $m\in M$. Here, we have used the operators
\begin{align*}
F_{\gamma}[(w,u)]&:=\sup_{\alpha\in\Lambda} \left\{\gamma^{\alpha}\left(-A^{\alpha}:D w -b^{\alpha}\cdot \nabla u  + c^{\alpha} u  -f^{\alpha}\right)\right\},\\ L_{\lambda}(w,u)&:= -\nabla\cdot w + \lambda u,
\end{align*}
acting on $(w,u)\in X$, and the positive constants
\begin{align*}
\sigma_1&:=\sigma_1(\delta):= 1-\frac{1}{2}\sqrt{1-\delta},\\ \sigma_2&:=\lambda\, \tilde{\sigma}_2(\delta):= \lambda\left(\frac{1-\sqrt{1-\delta}}{2}+\frac{1}{4\left(1-\sqrt{1-\delta}\right)}  \right).
\end{align*}
We proceed by showing well-posedness of this mixed formulation.

\subsection{Well-posedness of the mixed formulation}

We define a norm on the function space $X=W_{\mathrm{per}}(Y;\R^n)\times H^1_{\mathrm{per}}(Y)$ by
\begin{align*}
\vertiii{(w,u)}_{\lambda}^2:= \|Dw\|_{L^2(Y)}^2+2\lambda \|\nabla u\|_{L^2(Y)}^2+\lambda^2 \|u\|_{L^2(Y)}^2,\qquad (w,u)\in X.
\end{align*}
It is easy to verify that this does indeed define a norm on $X$. We observe that there holds
\begin{align}\label{Maxwell-type}
\|Dw\|_{L^2(Y)}^2 = \|\mathrm{rot}(w)\|_{L^2(Y)}^2+\|\nabla\cdot w\|_{L^2(Y)}^2\qquad \forall w\in H^1_{\mathrm{per}}(Y;\R^n),
\end{align}
which follows from the formal calculation (using integration by parts twice)
\begin{align*}
\int_Y \lvert Dw\rvert^2 - \int_Y \lvert \mathrm{rot}(w)\rvert^2 =\sum_{i,j=1}^n \int_Y  \partial_i w_j\, \partial_j w_i = \sum_{i,j=1}^n \int_Y  \partial_i w_i\, \partial_j w_j = \int_Y \lvert\nabla\cdot w\rvert^2,
\end{align*}
and a density argument. Note that compared to the usual Maxwell-type inequality \cite{CDa99}, we have obtained the equality \eqref{Maxwell-type} thanks to periodicity. Next, we derive two preliminary estimates.

\begin{lmm}[Preliminary estimates]\label{Preliminary estimates}
Let $(w,u),(w',u')\in X$ and $\rho\in (0,2)$. Then, there holds
\begin{align}\label{key Cordes implication}
\|F_{\gamma}[(w,u)]-F_{\gamma}[(w',u')]-L_{\lambda}(w-w',u-u')\|_{L^2(Y)}\leq \sqrt{1-\delta}\,\vertiii{(w-w',u-u')}_{\lambda},
\end{align}
and we have the Miranda--Talenti-type estimate
\begin{align}\label{MT}
\frac{2-\rho}{2}\vertiii{(w,u)}_{\lambda}^2\leq \|\mathrm{rot}(w)\|_{L^2(Y)}^2 + \|L_{\lambda}(w,u)\|_{L^2(Y)}^2 +\frac{\lambda}{\rho}\|\nabla u - w\|_{L^2(Y)}^2.
\end{align}
\end{lmm}

With these estimates in hand, we can proceed with showing essential properties of the maps $a$ and $b$, including monotonicity, Lipschitz continuity and an inf-sup condition, which will allow us to show well-posedness of the mixed formulation.

\begin{lmm}[Monotonicity, Lipschitz continuity and inf-sup condition]\label{Mon and Lip}
We have the following properties:
\begin{itemize}
\item[(i)] Monotonicity: For any $(w,u),(w',u')\in X$, we have
\begin{align*}
C_M\vertiii{(w-w',u-u')}_{\lambda}^2\leq a\left((w,u),(w-w',u-u')  \right)-a\left((w',u'),(w-w',u-u')  \right)
\end{align*}
with the monotonicity constant $C_M:=\frac{1}{4}\left(1-\sqrt{1-\delta}\right)>0$.
\item[(ii)] Lipschitz continuity: For any $(w,u),(w',u'),(z,v)\in X$, we have
\begin{align}\label{Lips}
\left\lvert a\left((w,u),(z,v)  \right)-a\left((w',u'),(z,v)  \right)\right\rvert\leq C_L \vertiii{(w-w',u-u')}_{\lambda} \vertiii{(z,v)}_{\lambda}
\end{align}
with the Lipschitz constant $C_L:= 2+\sqrt{2}\sqrt{1-\delta}+\sigma_1(\delta) +\tilde{\sigma}_2(\delta)\left(\frac{1}{2}+\frac{n}{\pi^2}\lambda \right) >0$.
\item[(iii)] Inf-sup condition: We have
\begin{align}\label{infsup}
\inf_{m'\in M\backslash \{0\}} \sup_{(w',u')\in X\backslash\{0\}} \frac{b(m',(w',u'))}{\|\nabla m'\|_{L^2(Y)}\,\vertiii{(w',u')}_{\lambda}}\geq c_b
\end{align}
with the inf-sup constant $c_b:=\lambda^{-\frac{1}{2}}\left(2+\frac{n}{\pi^2}\lambda\right)^{-\frac{1}{2}}>0$.
\end{itemize}
\end{lmm}

\begin{rmrk}[Local Lipschitz estimate]\label{local Lip}
Similarly, one obtains the local Lipschitz estimate
\begin{align*}
&\left\lvert a_{I}\left((w,u),(z,v)  \right)-a_{I}\left((w',u'),(z,v)  \right)\right\rvert\\ &\qquad\qquad\leq C_L'\left(\vertiii{(w-w',u-u')}_{\lambda,I}+\|w-w'\|_{L^2(I)}  \right)\left(\vertiii{(z,v)}_{\lambda,I}+\|z\|_{L^2(I)}  \right) 
\end{align*}
for all $(w,u),(w',u'),(z,v)\in X$ and any open $I\subset Y$ with a constant $C_L'=C_L'(\delta,\lambda,n)>0$. Here, the subscript $I$ in $a_{I}$ and $\vertiii{\cdot}_{\lambda,I}$ denotes that integrals in the corresponding definitions are taken over the set $I$.
\end{rmrk}

Now we are in a position to state the well-posedness of the mixed formulation, i.e., the existence and uniqueness of a solution $(m,(w,u))\in M\times X$ to \eqref{Mixed formulation}.

\begin{thrm}[Well-posedness of the mixed formulation]\label{wp of mixed}
The mixed formulation \eqref{Mixed formulation} admits a unique solution $(m,(w,u))\in M\times X$. Further, $m=0$, $u\in H^2_{\mathrm{per}}(Y)$ with $\nabla u=w$ and $u$ is the solution to \eqref{HJB equ}.
\end{thrm}

\subsection{The discrete mixed formulation}\label{subsec the discrete}

We take closed linear subspaces $W_h\subset W_{\mathrm{per}}(Y;\R^n)$, $U_h\subset H^1_{\mathrm{per}}(Y)$, $M_h\subset U_h\cap M$ (recall that $M\subset W_{\mathrm{per}}(Y)$), and define
\begin{align*}
X_h:= W_h\times U_h\subset X.
\end{align*}
We then define the discrete mixed formulation as the following problem: Find $m_h\in M_h$ and $(w_h,u_h)\in X_h$ such that
\begin{align}\label{Discrete mixed formulation}
\left\{ \begin{aligned} a\left((w_h,u_h),(w'_h,u'_h) \right)+ b(m_h,(w'_h,u'_h)) &= 0 &\quad& \forall \,(w'_h,u'_h)\in X_h,\\ b(m'_h,(w_h,u_h))&=0 &\quad& \forall\, m'_h\in M_h.\end{aligned}\right.
\end{align}
We note that we have boundedness of $b$ and a discrete inf-sup condition.
\begin{lmm}[Boundedness of $b$ and discrete inf-sup condition]\label{lmm bddness and disinfsup}
For any $(m',(w',u'))\in M\times X$, we have
\begin{align*}
b(m',(w',u'))\leq C_b \|\nabla m'\|_{L^2(Y)}\vertiii{(w',u')}_{\lambda}
\end{align*}
with the constant $C_b:=\lambda^{-\frac{1}{2}}\left(\frac{1}{2}+\frac{n}{\pi^2}\lambda\right)^{\frac{1}{2}}>0$. Further, the discrete inf-sup condition
\begin{align}\label{dis infsup}
\inf_{m'_h\in M_h\backslash \{0\}} \sup_{(w'_h,u'_h)\in X_h\backslash\{0\}} \frac{b(m'_h,(w'_h,u'_h))}{\|\nabla m'_h\|_{L^2(Y)}\,\vertiii{(w'_h,u'_h)}_{\lambda}}\geq c_b
\end{align}
holds with $c_b>0$ as in Lemma \ref{Mon and Lip} (iii).
\end{lmm}

It follows that we have well-posedness of the discrete mixed formulation analogously to Theorem \ref{wp of mixed}. We also obtain an error bound.

\begin{thrm}[Well-posedness and error bound]\label{Thm error bd}
There exists a unique solution $(m_h,(w_h,u_h))\in M_h\times X_h$ to the discrete mixed formulation \eqref{Discrete mixed formulation}. Further, we have
\begin{align*}
\vertiii{(w-w_h,u-u_h)}_{\lambda}\leq C_e \inf_{(w'_h,u'_h)\in X_h}\vertiii{(w-w'_h,u-u'_h)}_{\lambda},
\end{align*}
where $(m,(w,u))\in M\times X$ denotes the solution to \eqref{Mixed formulation} and $C_e=C_e(\delta,\lambda,n)>0$ is the constant
\begin{align*}
C_e:=2\frac{C_L}{C_M}\left(1+\frac{C_b}{c_b}\right)
\end{align*}
with $C_L,C_M,C_b,c_b>0$ from Lemmata \ref{Mon and Lip} and \ref{lmm bddness and disinfsup}.
\end{thrm}

\begin{rmrk}\label{rk explicit Ce}
Note that the error constant $C_e=C_e(\delta,\lambda,n)$ is monotonically increasing in $\lambda$.
\end{rmrk}

Besides this \textit{a priori} error bound, the monotonicity property from Lemma \ref{Mon and Lip} allows us to obtain an \textit{a posteriori} error bound.

\begin{thrm}[\textit{a posteriori} error bound and efficiency]\label{thrm a posteriori}
For the solution $(m,(w,u))\in M\times X$ to the mixed formulation \eqref{Mixed formulation} and the solution $(m_h,(w_h,u_h))\in M_h\times X_h$ to the discrete mixed formulation \eqref{Discrete mixed formulation}, we have the error bound
\begin{align*}
\vertiii{(w-w_h,u-u_h)}_{\lambda}^2\leq 2\,C_M^{-1} \left(C_M^{-1}\left\|F_{\gamma}[(w_h,u_h)]\right\|_{L^2(Y)}^2 + \sigma_1 \left\|\mathrm{rot}(w_h)\right\|_{L^2(Y)}^2 + \sigma_2\left\|w_h-\nabla u_h\right\|_{L^2(Y)}^2  \right)
\end{align*}
and the efficiency estimate
\begin{align*}
&\frac{1}{2}\left\|F_{\gamma}[(w_h,u_h)]\right\|_{L^2(Y)}^2 + \sigma_1\left\|\mathrm{rot}(w_h)\right\|_{L^2(Y)}^2 +\sigma_2 \left\|w_h-\nabla u_h\right\|_{L^2(Y)}^2 \\ &\qquad\qquad\leq \left(C_L+\frac{1-\delta}{2}\right)\vertiii{(w-w_h,u-u_h)}_{\lambda}^2,
\end{align*}
where $C_M,C_L>0$ are the constants from Lemma \ref{Mon and Lip}.
\end{thrm}

\begin{rmrk}[Local efficiency]\label{rmrk local efficiency}
Similarly, one obtains the local efficiency estimate
\begin{align*}
\frac{1}{2}\left\|F_{\gamma}[(w_h,u_h)]\right\|_{L^2(I)}^2&+ \sigma_1\left\|\mathrm{rot}(w_h)\right\|_{L^2(I)}^2 +\sigma_2 \left\|w_h-\nabla u_h\right\|_{L^2(I)}^2 \\ &\leq \left(2C_L'+\frac{1-\delta}{2}\right)\left(\vertiii{(w-w_h,u-u_h)}_{\lambda,I}^2+ \left\|w-w_h\right\|_{L^2(I)}^2\right)
\end{align*}
for any open $I\subset Y$, where $C_L'>0$ is the constant from Remark \ref{local Lip}.
\end{rmrk}

\section{Numerical Homogenization of HJB Equations}\label{Section NumHom}

\subsection{Framework}\label{subsec fra}

Let $\Omega\subset \R^n$ be a bounded convex domain in dimension $n\in\{2,3\}$ and let $\Lambda$ be a compact metric space. For $\eps>0$ small, we consider problems of the form
\begin{align}\label{equ}
\left\{ \begin{aligned}\sup_{\alpha\in\Lambda} \left\{ -A^{\alpha}\left(\,\cdot\,,\frac{\cdot}{\eps}\right):D^2 u_{\eps} - b^{\alpha}\left(\,\cdot\,,\frac{\cdot}{\eps}\right)\cdot \nabla u_{\eps} + u_{\eps} -f^{\alpha}\left(\,\cdot\,,\frac{\cdot}{\eps}\right)\right\} &= 0 &\quad &\text{in }\Omega,\\ \hfill u_{\eps}&= 0 &\quad  &\text{on }\partial\Omega,\end{aligned}\right.
\end{align}
where we assume that the functions
\begin{align*}
A=(a_{ij})_{1\leq i,j\leq n}:\bar{\Omega}\times\R^n\times \Lambda&\rightarrow \calS^{n\times n},& (x,y,\alpha)&\mapsto A(x,y,\alpha)=:A^{\alpha}(x,y),\\
b=(b_i)_{1\leq i\leq n}:\bar{\Omega}\times\R^n\times\Lambda&\rightarrow \R^n,& (x,y,\alpha)&\mapsto b(x,y,\alpha)=:b^{\alpha}(x,y),\\
f:\bar{\Omega}\times \R^n\times\Lambda&\rightarrow \R,& (x,y,\alpha)&\mapsto f(x,y,\alpha)=:f^{\alpha}(x,y)
\end{align*}
satisfy the following assumptions:
\begin{itemize}
\item[(i)] Continuity: $A,\,b,\,f$ are continuous on $\bar{\Omega}\times \R^n\times\Lambda$;
\item[(ii)] Periodicity: $A^{\alpha}(x,\cdot),\,b^{\alpha}(x,\cdot),\,f^{\alpha}(x,\cdot)$ are $Y$-periodic for fixed $\alpha\in\Lambda$ and $x\in\bar{\Omega}$;
\item[(iii)] Regularity: $A^{\alpha},\,b^{\alpha},\,f^{\alpha}$ are Lipschitz on $\bar{\Omega}\times\R^n$ uniformly in $\alpha\in\Lambda$;
\item[(iv)] Ellipticity: There exist $\zeta_1,\zeta_2>0$ such that $\zeta_1 \lvert \xi\rvert^2\leq  A\xi \cdot \xi\leq \zeta_2 \lvert \xi\rvert^2$ in $\bar{\Omega}\times\R^n\times \Lambda$ for all $\xi\in\R^n$.
\end{itemize}
Further, it is assumed that the Cordes condition
\begin{align}\label{Cordes}
\lvert A\rvert^2+\frac{\lvert b\rvert^2}{2\lambda} + \frac{1}{\lambda^2}\leq \frac{1}{n+\delta} \left(\mathrm{tr}( A) + \frac{1}{\lambda}  \right)^2   \qquad\text{in }\bar{\Omega}\times\R^n\times \Lambda
\end{align}
holds for some constants $\lambda>0$ and $\delta\in (0,1)$. Then, we have well-posedness in the sense of strong solutions; see \cite{SS14}.
\begin{thrm}[Existence and uniqueness of strong solutions]\label{well posedness}
In this situation, for any given $\eps>0$, there exists a unique strong solution $u_{\eps}\in H^2(\Omega)\cap H^1_0(\Omega)$ to \eqref{equ}.
\end{thrm}

\begin{rmrk}
Problems involving a non-constant zeroth-order coefficient, i.e., problems of the form
\begin{align*}
\left\{ \begin{aligned}\sup_{\alpha\in\Lambda} \left\{ -A^{\alpha}\left(\,\cdot\,,\frac{\cdot}{\eps}\right):D^2 v_{\eps} - b^{\alpha}\left(\,\cdot\,,\frac{\cdot}{\eps}\right)\cdot \nabla v_{\eps} + c^{\alpha}\left(\,\cdot\,,\frac{\cdot}{\eps}\right) v_{\eps} -f^{\alpha}\left(\,\cdot\,,\frac{\cdot}{\eps}\right)\right\} &= 0 &\quad &\text{in }\Omega,\\ \hfill v_{\eps}&= 0 &\quad  &\text{on }\partial\Omega,\end{aligned}\right.
\end{align*}
with $c^{\alpha}$ satisfying the same assumptions as the components of $b^{\alpha}$, and additionally $\inf_{\bar{\Omega}\times\R^n\times\Lambda}c>0$, can be reduced to a problem of the form \eqref{equ}. This is due to the fact that division by $c^{\alpha}(x,x/\eps)$ inside the argument of the supremum does not change the set of strong solutions.
\end{rmrk}

\subsection{Homogenization}\label{subsec homogenization}

In this section, we briefly recall known homogenization results from the literature. Let us start by recalling one of the several equivalent definitions of a viscosity solution; see \cite{Lii83}.

\begin{dfntn}[Viscosity solution]\label{dfntn visc}
Let $\Omega\subset\R^n$ be open and $F:\Omega\times\R\times\R^n\times\calS^{n\times n}\rightarrow \R$ be continuous. A continuous function $u:\Omega\rightarrow \R$, $u\in C(\bar{\Omega})$, is called a viscosity solution to the equation
\begin{align*}
F(x,u,\nabla u,D^2 u)=0\quad\text{in }\Omega,
\end{align*}
if for any $\phi\in C^2(\Omega)$ there holds
\begin{align*}
x_0\in\Omega \text{ local maximum point of }u-\phi \;\Longrightarrow\; F(x_0,u(x_0),\nabla \phi(x_0),D^2\phi(x_0))\leq 0,\\
x_0\in\Omega \text{ local minimum point of }u-\phi \;\Longrightarrow\; F(x_0,u(x_0),\nabla \phi(x_0),D^2\phi(x_0))\geq 0.
\end{align*}
\end{dfntn}

For an overview of the theory of viscosity solutions for second-order equations we refer the reader to \cite{CIL92}. Note that the strong solution $u_{\eps}\in H^2(\Omega)\cap H^1_0(\Omega)$ to \eqref{equ} belongs to $C(\bar{\Omega})$ in dimensions $n\in\{2,3\}$ and a natural question to ask is whether $u_{\eps}$ is a viscosity solution. In fact, it is known that if one has regularity $u_{\eps}\in W^{2,n}_{\mathrm{loc}}(\Omega)$, then $u_{\eps}$ is a viscosity solution to \eqref{equ}; see \cite{CCK96,Lii83,Lio83}. We also note that the viscosity solution to \eqref{equ} is unique; see \cite{Ish89}. We then have the following result; see \cite{Saf88}.

\begin{rmrk}[Regularity]\label{r:regularity}
Let $u_{\eps}\in H^2(\Omega)\cap H^1_0(\Omega)$ be the unique strong solution to \eqref{equ} given by Theorem \ref{well posedness}. Then
\begin{align*}
u_{\eps}\in C^{2,\tilde{\alpha}}(\Omega)\cap C(\bar{\Omega})
\end{align*}
for some $\tilde{\alpha}>0$ and $u_{\eps}$ is the unique viscosity solution to \eqref{equ}. Further, if $\partial\Omega\in C^{2,\beta}$ for some $\beta>0$, then $u_{\eps}\in C^{2,\tilde{\alpha}}(\bar{\Omega})$ for some $\tilde{\alpha}>0$.
\end{rmrk}

With this observation in hand, we can use the well-known homogenization results for viscosity solutions; see \cite{CSW05,Eva89,Eva92}.

\begin{thrm}[Homogenization of HJB problems]\label{thm Hom HJB}
The solution $u_{\eps}$ to \eqref{equ} converges uniformly on $\bar{\Omega}$ to the viscosity solution $u_0\in C(\bar{\Omega})$ of
\begin{align}\label{hom}
\left\{ \begin{aligned}u_0 + H(x,\nabla u_0, D^2 u_0) &= 0 &\quad &\text{in }\Omega,\\ \hfill u_0&= 0 &\quad  &\text{on }\partial\Omega,\end{aligned}\right.
\end{align}
with an effective Hamiltonian $H:\bar{\Omega}\times \R^n\times \calS^{n\times n}\rightarrow \R$ defined as follows: For given $(s,p,R)\in \bar{\Omega}\times \R^n\times \calS^{n\times n}$ we define $H(s,p,R)\in\R$ to be the unique real number such that there exists a function $v=v(\cdot\,;s,p,R)\in C(\R^n)$, a so-called corrector, that is a viscosity solution to
\begin{align}\label{corr v}
\sup_{\alpha\in\Lambda} \left\{-A_s^{\alpha}:D^2 v - g^{\alpha}_{s,p,R}\right\} = H(s,p,R) \quad \text{in }\R^n,\qquad v(\cdot\,;s,p,R) \text{ is $Y$-periodic},
\end{align}
where $A_s^{\alpha}(y):=A^{\alpha}(s,y)$ and $g^{\alpha}_{s,p,R}(y):=A^{\alpha}(s,y): R+ b^{\alpha}(s,y)\cdot p +f^{\alpha}(s,y)$ for $y\in \R^n$, $\alpha\in \Lambda$.
\end{thrm}

Let us note that rates for the convergence of $u_{\eps}$ to the homogenized solution $u_0$ have been derived for the whole space problem in \cite{CM09}.

The effective Hamiltonian can also be obtained through a limit of ergodic approximations, the so-called approximate correctors; see \cite{AB10} and the references therein. For $(s,p,R)\in \bar{\Omega}\times \R^n\times \calS^{n\times n}$ and $\sigma>0$, the approximate corrector $v^{\sigma}=v^{\sigma}(\cdot\,;s,p,R)\in C(\R^n)$ is defined to be the viscosity solution to
\begin{align}\label{ergodic alternative}
\sigma v^{\sigma}+\sup_{\alpha\in\Lambda} \left\{-A_s^{\alpha}:D^2 v^{\sigma} -g^{\alpha}_{s,p,R}\right\} = 0 \quad \text{in }\R^n,\qquad v^{\sigma}(\cdot\,;s,p,R) \text{ is $Y$-periodic}.
\end{align}

\begin{rmrk}[Regularity of approximate correctors]\label{Rmrk reg approx corr}
The viscosity solution $v^{\sigma}=v^{\sigma}(\cdot\,;s,p,R)\in C(\R^n)$ to \eqref{ergodic alternative} is in fact a classical solution $v^{\sigma}(\cdot\,;s,p,R)\in C^2(\R^n)$. Further, there exists $\tilde{\alpha}\in (0,1)$ such that
\begin{align*}
\left\| \sigma v^{\sigma}(\cdot\,;s,p,R)\right\|_{C(\R^n)}+\left\| v^{\sigma}(\cdot\,;s,p,R)-v^{\sigma}(0\,;s,p,R) \right\|_{C^{2,\tilde{\alpha}}(\R^n)}\lesssim 1+\lvert p\rvert +\lvert R\rvert
\end{align*}
for all $(s,p,R)\in \bar{\Omega}\times \R^n\times \calS^{n\times n}$; see \cite{AB01,CM09}.
\end{rmrk}

The value $H(s,p,R)\in \R$ for the effective Hamiltonian at the point $(s,p,R)$ is then the uniform limit of the sequence $\left\{-\sigma v^{\sigma}\right\}_{\sigma>0}$ as $\sigma\rightarrow 0$; see \cite{CM09}.

\begin{lmm}[Properties of the effective Hamiltonian]\label{Lmm: prop of H}
The following statements hold true.
\begin{itemize}
\item[(i)] The sequence $\left\{-\sigma v^{\sigma}(\cdot\,;s,p,R)\right\}_{\sigma>0}$ converges uniformly to the constant value $H(s,p,R)$ with
\begin{align*}
\left\| -\sigma v^{\sigma}(\cdot\,;s,p,R)-H(s,p,R)\right\|_{\infty} \lesssim \sigma \left(1+\lvert p\rvert +\lvert R\rvert \right)
\end{align*}
for all $(s,p,R)\in \bar{\Omega}\times \R^n\times \calS^{n\times n}$.
\item[(ii)] The effective Hamiltonian $H=H(s,p,R)$ is uniformly elliptic, it is convex in $R$, and we have
\begin{align*}
\left\lvert H(s,p_1,R_1)-H(s,p_2,R_2)\right\rvert &\lesssim \lvert p_1-p_2\rvert+\lvert R_1-R_2\rvert,\\
\left\lvert H(s_1,p,R)-H(s_2,p,R)\right\rvert &\lesssim \lvert s_1-s_2\rvert \left(1+\lvert p\rvert+\lvert R\rvert\right),
\end{align*}
for any $s,s_1,s_2\in\bar{\Omega}$, $p,p_1,p_2\in\R^n$ and $R,R_1,R_2\in \calS^{n\times n}$.
\end{itemize}
\end{lmm}

Note that the properties of the approximate correctors from Remark \ref{Rmrk reg approx corr} and Lemma \ref{Lmm: prop of H} (i) allow passage to the limit $\sigma\rightarrow 0$ in \eqref{ergodic alternative} and guarantee the existence of a corrector $v(\cdot\,;s,p,R)\in C^2(\R^n)$ (i.e., a classical solution to \eqref{corr v}). We also note that the properties of the effective Hamiltonian from Lemma \ref{Lmm: prop of H} (ii) yield a regularity result for the homogenized solution as it is of the type of problems studied in \cite{Saf88}.

\begin{rmrk}[Regularity of the homogenized solution]\label{r:regularity hom sol}
The viscosity solution $u_0\in C(\bar{\Omega})$ to the homogenized problem \eqref{hom} satisfies
\begin{align*}
u_0\in C^{2,\tilde{\alpha}}(\Omega)\cap C(\bar{\Omega})
\end{align*}
for some $\tilde{\alpha}>0$. Further, if $\partial\Omega\in C^{2,\beta}$ for some $\beta>0$, then $u_0\in C^{2,\tilde{\alpha}}(\bar{\Omega})$ for some $\tilde{\alpha}>0$.
\end{rmrk}

\subsection{Approximation of the approximate corrector}\label{subsec mixed fem for appr corr}

We construct a mixed finite element method for the numerical approximation of the approximate corrector for fixed $(s,p,R)\in \bar{\Omega}\times \R^n\times \calS^{n\times n}$. For $\sigma\in (0,1)$ we consider the problem \eqref{ergodic alternative}, i.e., the problem of finding a strong solution $v^{\sigma}=v^{\sigma}(\cdot\,;s,p,R)$ to
\begin{align}\label{v sigma prob}
\sup_{\alpha\in\Lambda} \left\{-A_s^{\alpha}:D^2 v^{\sigma}+\sigma v^{\sigma} -g^{\alpha}_{s,p,R}\right\} = 0 \quad \text{in }\R^n,\qquad v^{\sigma}(\cdot\,;s,p,R)\in H^2_{\mathrm{per}}(Y).
\end{align}
Recall the notation $A_s^{\alpha}(y):=A_s(y,\alpha):=A(s,y,\alpha)$ and
\begin{align*}
g_{s,p,R}^{\alpha}(y):=g_{s,p,R}(y,\alpha):=A^{\alpha}(s,y): R+ b^{\alpha}(s,y)\cdot p +f^{\alpha}(s,y)
\end{align*}
for $y\in\R^n$ and $\alpha\in\Lambda$ from Theorem \ref{thm Hom HJB}.

Note that $g_{s,p,R}:\R^n\times \Lambda\rightarrow \R$ is continuous, and that $g^{\alpha}_{s,p,R}$ is $Y$-periodic for fixed $\alpha\in \Lambda$ and Lipschitz on $\R^n$ uniformly in $\alpha$. We also note that we have the Cordes condition \eqref{Cordes}, which yields
\begin{align}\label{Cordes lambda sigma}
\lvert A_s\rvert^2+ \frac{\sigma^2}{\lambda_{\sigma}^2}\leq \frac{1}{n+\delta}\left(\mathrm{tr}( A_s) + \frac{\sigma}{\lambda_{\sigma}}  \right)^2   \qquad\text{in }\R^n\times \Lambda,
\end{align}
where $\lambda_{\sigma}>0$ is given by
\begin{align*}
\lambda_{\sigma}:= \sigma\lambda.
\end{align*}
The corresponding scaling function $\gamma^{\alpha}(y):=\gamma(y,\alpha)$ is defined by (compare with \eqref{gamma first def})
\begin{align*}
\gamma:= \left(\lvert A_s\rvert^2+ \frac{\sigma^2}{\lambda_{\sigma}^2}  \right)^{-1} \left(\mathrm{tr}( A_s) + \frac{\sigma}{\lambda_{\sigma}}\right).
\end{align*}

Observe that \eqref{Cordes lambda sigma} is the Cordes condition \eqref{Cordes periodic} for the problem \eqref{v sigma prob} with Cordes constants $\delta$ and $\lambda_{\sigma}$. Therefore, Theorem \ref{thm well-pos} ensures the well-posedness of the problem \eqref{v sigma prob}, i.e., existence and uniqueness of a strong $Y$-periodic solution. We apply the mixed finite element method from Section \ref{subsec the discrete} to problem \eqref{v sigma prob} to obtain
an approximation.

The scheme from Section \ref{Section FE schemes for periodic HJB} applied to the problem \eqref{v sigma prob} yields an approximation $(m_h^{\sigma},(w_h^{\sigma},v_h^{\sigma}))\in M_h\times X_h$ with error bound
\begin{align}\label{appr mix}
\vertiii{(\nabla v^{\sigma}-w_h^{\sigma},v^{\sigma}-v_h^{\sigma})}_{\lambda_{\sigma}}\leq C_e(\delta,\lambda_{\sigma},n) \inf_{(w'_h,u'_h)\in X_h}\vertiii{(\nabla v^{\sigma}-w'_h, v^{\sigma}-u'_h)}_{\lambda_{\sigma}},
\end{align}
and we have that $C_e(\delta,\lambda_{\sigma},n)\leq C_e(\delta,\lambda,n)$ for all $\sigma\in (0,1)$.  For a shape-regular triangulation $\calT_h$ on $Y$, denoting the Lagrange finite element space
of degree $q\in\N$ over the triangulation by $\calS^q(\calT_h)$, we obtain the following approximation result:

\begin{thrm}[Error bound for the approximate corrector]\label{thrm error bd mixed}
For $\sigma\in (0,1)$, if we have $v^{\sigma}=v^{\sigma}(\cdot\,;s,p,R)\in H^{2+r}(Y)$ for some $r\geq 0$ and the choice
\begin{align*}
X_h:= \left(\calS^q(\calT_h;\R^n)\cap W_{\mathrm{per}}(Y;\R^n) \right)\times \left(\calS^l(\calT_h)\cap H^1_{\mathrm{per}}(Y)\right)
\end{align*}
for some $q,l\in \N$ and a shape-regular triangulation $\calT_h$ on $Y$ (consistent with the requirement of periodicity), we find that
\begin{align*}
\vertiii{(\nabla v^{\sigma}-w_h^{\sigma},v^{\sigma}-v_h^{\sigma})}_{\lambda_{\sigma}}\leq C h^{\min\{r,q,l\}} \|\nabla v^{\sigma}\|_{H^{1+r}(Y)}
\end{align*}
for $h>0$ sufficiently small, with the constant $C>0$ only depending on $\delta,\lambda,n$ and interpolation constants.
\end{thrm}

\begin{rmrk}
The proof yields that the error constant can be taken to be
\begin{align*}
C:=C_e(\delta,\lambda,n)C_i(1+\lambda),
\end{align*}
where $C_i$ is a constant arising from interpolation inequalities.
\end{rmrk}

\subsection{Approximation of the effective Hamiltonian}\label{subsec app of eff ham}

The approximation of the approximate corrector from the previous section allows us to obtain an approximation to the effective Hamiltonian as follows.

First, we note that with $\tilde{\alpha}\in (0,1)$ from Remark \ref{Rmrk reg approx corr} we have that, for any $r\in [0,\tilde{\alpha})$, there holds
\begin{align*}
\sup_{\sigma\in (0,1)}\| \nabla v^{\sigma}(\cdot\,;s,p,R)\|_{H^{1+r}(Y)}\lesssim \sup_{\sigma\in (0,1)}\|\nabla v^{\sigma}(\cdot\,;s,p,R)\|_{C^{1,\tilde{\alpha}}(\R^n)}\lesssim 1+\lvert p\rvert +\lvert R\rvert,
\end{align*}
uniformly in $\sigma$. Using the error bound from Theorem \ref{thrm error bd mixed}, we deduce that
\begin{align*}
\vertiii{(\nabla v^{\sigma}-w_h^{\sigma},v^{\sigma}-v_h^{\sigma})}_{\lambda_{\sigma}}\lesssim h^{\min\{r,q,l\}} \|\nabla v^{\sigma}\|_{H^{1+r}(Y)}\lesssim h^{\min\{r,q,l\}}(1+\lvert p\rvert +\lvert R\rvert),
\end{align*}
with a constant independent of $\sigma$ and the choice of $(s,p,R)$. In particular, by the definition of $\vertiii{\cdot}_{\lambda_{\sigma}}$, we have
\begin{align}\label{helperror}
\|\sigma v^{\sigma}-\sigma v_h^{\sigma}\|_{L^2(Y)}\lesssim h^{\min\{r,q,l\}} (1+\lvert p\rvert +\lvert R\rvert).
\end{align}
We then define the approximated effective Hamiltonian as
\begin{align}\label{appreffHam}
H_{\sigma,h}:\bar{\Omega}\times \R^n\times \calS^{n\times n}\rightarrow\R,\qquad  H_{\sigma,h}(s,p,R):= -\sigma \int_Y v^{\sigma}_{h}(\cdot\,;s,p,R).
\end{align}
Then, the following approximation result holds.

\begin{thrm}[Approximation of the effective Hamiltonian]\label{App of eff H}
Let $\sigma\in (0,1)$ and $(w^{\sigma}_h,v^{\sigma}_h)\in X_h$ as in Theorem \ref{thrm error bd mixed}. Further let $H_{\sigma,h}$ be defined as in \eqref{appreffHam}. Then, for $(s,p,R)\in \bar{\Omega}\times \R^n\times \calS^{n\times n}$, we have the error bound
\begin{align*}
\left\lvert H_{\sigma,h}(s,p,R)- H(s,p,R)\right\rvert \lesssim \left(h^r + \sigma\right)(1+\lvert p\rvert +\lvert R\rvert)
\end{align*}
for any $r\in [0,\tilde{\alpha})$ with $\tilde{\alpha}\in (0,1)$ from Remark \ref{Rmrk reg approx corr}. More generally, for fixed $(s,p,R)\in \bar{\Omega}\times \R^n\times \calS^{n\times n}$, we have
\begin{align*}
\left\lvert H_{\sigma,h}(s,p,R)- H(s,p,R)\right\rvert=\calO\left(h^{\min\{r,q,l\}} + \sigma\right)
\end{align*}
for any $r\geq 0$ such that $\{\|\nabla v^{\sigma}(\cdot\,;s,p,R)\|_{H^{1+r}(Y)}\}_{\sigma\in(0,1)}$ is uniformly bounded.
\end{thrm}

\section{Numerical Experiments}\label{Sec Num Exp}

\subsection{Set-up}

We consider the problem of approximating the solution $u_\varepsilon$ to the HJB equation
\begin{equation*}
 \left\{
 \begin{aligned}
  \sup_{\alpha\in\Lambda}
   \left\{
    -A^\alpha\left(\frac{\cdot}{\varepsilon}\right):D^2u_\varepsilon
    +u_\varepsilon -1
   \right\}
   &=0\quad&&\text{in }\Omega,
   \\
   u_\varepsilon&=0 &&\text{on }\partial\Omega,
 \end{aligned}
 \right.
\end{equation*}
where $\Omega:=(0,1)^2\subseteq\mathbb R^2$ is the unit square and $\Lambda := [0,1]$. The coefficient $A$ has the structure
\begin{equation*}
 A:\mathbb R^2\times\Lambda\to \mathcal S^{2\times 2},
 \qquad
 A(y,\alpha):=A^\alpha(y):=\left(a_0(y)+\alpha a_1(y)\right)B
\end{equation*}
for $Y$-periodic functions $a_0,a_1:\mathbb R^2\to(0,\infty)$ and a symmetric positive definite matrix $B\in\mathcal S^{2\times 2}$. The homogenized problem \eqref{hom} is then given by
\begin{align*}
\left\{ \begin{aligned}u_0 + H(D^2 u_0) &= 0 &\quad &\text{in }\Omega,\\ \hfill u_0&= 0 &\quad  &\text{on }\partial\Omega,\end{aligned}\right.
\end{align*}
and an explicit expression for the effective Hamiltonian, according to \cite[Section 2.2]{FO18},
is given by
\begin{equation}\label{e:FinlayOberman_formula}
 H:\mathcal S^{2\times2}\to\mathbb R,
 \qquad
 H(R)=\max\left\{
  -\left(\int_Y\frac{1}{a_0}\right)^{-1}B:R,
  -\left(\int_Y\frac{1}{a_0+a_1}\right)^{-1}B:R
 \right\}-1.
\end{equation}
Explicitly, we choose in our numerical experiments
\begin{equation*}
 B:=\begin{pmatrix}
    2&-1\\-1&4
   \end{pmatrix}
   ,\quad
   a_0\equiv1
   ,\quad
   a_1(y_1,y_2):=\sin^2(2\pi y_1)\cos^2(2\pi y_2) +1.
\end{equation*}

\subsection{First experiment: Approximation of the effective Hamiltonian at a point}

Our objective in the first numerical experiment is to investigate the approximation of the effective Hamiltonian $H(R)$ by the numerically computed approximate Hamiltonian $H_{\sigma,h}(R)$ at some given point $R\in\mathcal S^{2\times 2}$. We choose
\begin{align*}
R:=\begin{pmatrix}-2&1\\1&-3\end{pmatrix}
\end{align*}
as a negative definite matrix so that the maximum in \eqref{e:FinlayOberman_formula} is realized by the term involving the harmonic integral mean of $a_0+a_1$ (i.e., the term involving $\left[\int_Y (a_0+a_1)^{-1}\right]^{-1}$).
For our discretization, we choose a continuous piecewise affine discretization with $q=l=1$ and $M_h:=\{0\}$. 
In order to compare the experimental results with the theoretical bound of Theorem~\ref{App of eff H}, we consider convergence in $h$ and $\sigma$ separately. We test convergence with respect to $h$ by fixing a (sufficiently small)
value $\sigma=0.01$ and choosing a uniform mesh-refinement of the periodicity cell $Y=(0,1)^2$. Since the error bound for the approximate corrector from Theorem~\ref{thrm error bd mixed} is given in the norm $\vertiii{\cdot}_{\lambda_{\sigma}}$, we first numerically test the convergence rate predicted by Theorem~\ref{thrm error bd mixed}.
The exact approximate corrector $v^\sigma$ is unknown, and thus we instead compute the \textit{a posteriori} error estimator
$$
\eta(h):=
\left\|F_{\gamma}[(w_h,u_h)]\right\|_{L^2(Y)}^2
+ \sigma_1\left\|\mathrm{rot}(w_h)\right\|_{L^2(Y)}^2 
+\sigma_2 \left\|w_h-\nabla u_h\right\|_{L^2(Y)}^2,
$$
which is, up to a constant factor, equivalent to the error in Theorem~\ref{thrm error bd mixed}; see Theorem~\ref{thrm a posteriori} and Remark~\ref{rmrk local efficiency}. The convergence histories of $\eta/100$ and the relative error
$$
  \frac{|H_{\sigma,h}(R)-H(R)|}{|H(R)|}
$$
are displayed in Figure~\ref{f:convergence_corr_hrefinement}. As we are mainly interested in the rate of convergence, we plot $\eta/100$ so that both error quantities can be shown in the same diagram.

\begin{figure}[H]
\begin{tikzpicture}[scale=1.15]
  \begin{loglogaxis}
        [
            xlabel=$1/h$,
            ylabel style={rotate=0},
            ylabel={},
            xmin=1,xmax=1e2,
            ymin=1e-7,ymax=1e-1,
            legend style=
            {
                at={(1,.5)},
                anchor=west,
                draw=black,
                fill=none
            },
            cycle list name=black white,
%             legend columns =2,
%             legend cell align=left,
            grid=major,
            clip=false,
            scaled ticks=true
        ]
      \addplot table[x index =0,y index=1]{./numresults/corr_approx/etaVSinvh.dat};
      \addplot table[x index =0,y index=1]{./numresults/corr_approx/relerrVSinvh.dat};
      \legend
            {   $\eta/100$,
                $\frac{|H_{\sigma,h}(R)-H(R)|}{|H(R)|}$
            }
    \end{loglogaxis}
  \end{tikzpicture}
  \caption{Error estimator and 
           approximation error between $H(R)$ by $H_{\sigma,h}(R)$ 
           under mesh refinement with fixed $\sigma=0.01$.}
  \label{f:convergence_corr_hrefinement}
\end{figure}
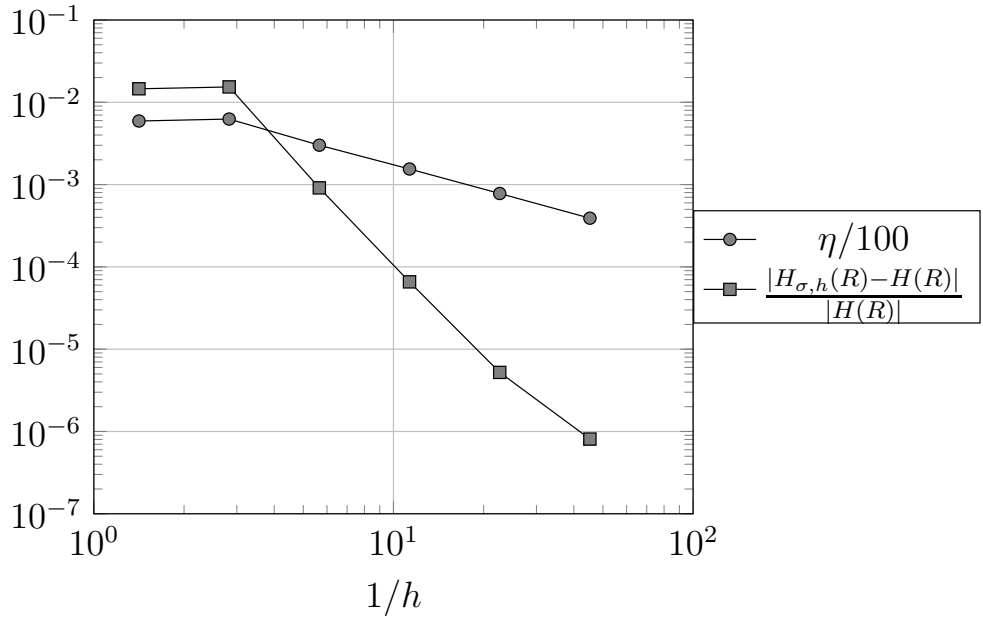

As expected from Theorem~\ref{thrm error bd mixed}, the error estimator is of order $h$, whereas we observe cubic convergence with respect to $h$ for the relative error of the effective Hamiltonian at the point $R$. This rate is higher than predicted by Theorem~\ref{App of eff H}, which is based on an error estimate in the norm $\vertiii{\cdot}_{\lambda_{\sigma}}$ and is therefore indeed expected to overestimate the actual error between $H_{\sigma,h}(R)$ and $H(R)$ related to the weaker integral functional from \eqref{appreffHam}.

Next, we test convergence with respect to $\sigma$ by fixing a fine mesh size $h=\sqrt{2}\times 2^{-7}$ and letting $\sigma$ vary from $2^4$ to $2^{-7}$. The convergence history of the relative error is displayed in Figure~\ref{f:convergence_corr_sigmarefinement}.

\begin{figure}[H]
\begin{tikzpicture}[scale=1.15]
  \begin{loglogaxis}
        [
            xlabel=$1/\sigma$,
            ylabel style={rotate=0},
            ylabel={relative error $\frac{|H_{\sigma,h}(R)-H(R)|}{|H(R)|}$},
            xmin=1e-2,xmax=1e3,
            ymin=1e-8,ymax=1e-3,
%             xticklabels = {$10^0$,$10^{-1}$,$10^{-2}$},
%             xtick = {1e0,1e1,1e2},
%             legend style=
%             {
%                 at={(0.5,1)},
%                 anchor=south,
%                 draw=black,
%                 fill=none
%             },
            cycle list name=black white,
%             legend columns =2,
%             legend cell align=left,
            grid=major,
            clip=false,
            scaled ticks=true
        ]
      \addplot table[x index =0,y index=1]{./numresults/corr_approx/relerrVSinvsigma.dat};
    \end{loglogaxis}
  \end{tikzpicture}
\caption{Approximation of $H(R)$ by $H_{\sigma,h}(R)$ for varying $\sigma$ with fixed mesh size $h=\sqrt{2}\times 2^{-7}$.}
  \label{f:convergence_corr_sigmarefinement}
\end{figure}
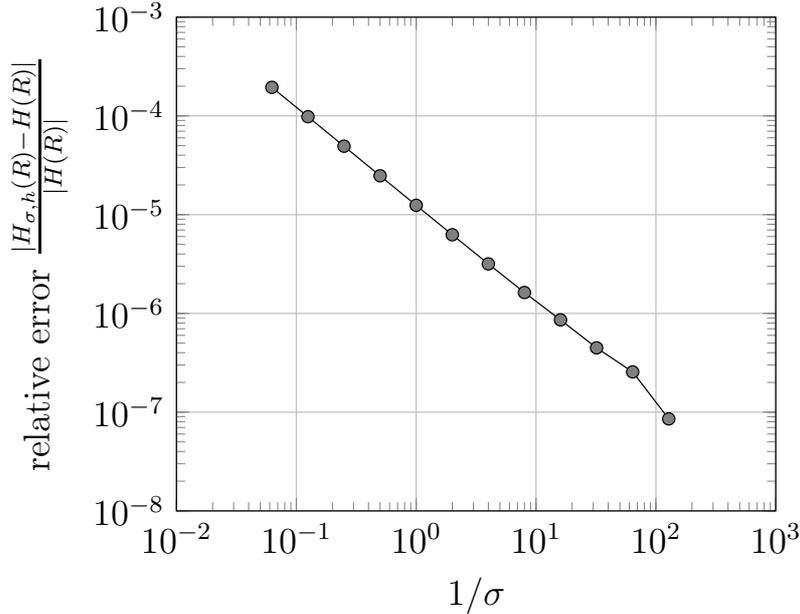

We observe linear convergence with respect to $\sigma$, which indicates that the bound in Theorem~\ref{App of eff H} is sharp in $\sigma$.

\subsection{Second experiment: Numerical approximation of the homogenized problem}

The second numerical experiment is devoted to the approximation of the effective problem \eqref{hom}. We first note that the discretization on the scales $\Omega$ and $Y$ leads to a two-scale approach. We denote the triangulation of $\Omega$ by $\mathcal T_h^\Omega$ with mesh size $h_\Omega$ and the triangulation of  $Y$ by $\mathcal T_h^Y$ with mesh size $h_Y$. In view of the regularity result from Remark~\ref{r:regularity hom sol}, we discretize the solution $u_0$ of this fully nonlinear equation by a least-squares approach, which is explained in the following. We discretize functions over $\Omega$ using the finite element space consisting of continuous piecewise affine functions
$$
 \mathcal S^1_0(\mathcal T_h^\Omega)
$$
satisfying a homogeneous Dirichlet boundary condition, and their gradients by vector-valued continuous piecewise affine finite elements
$$
 \mathcal S^1(\mathcal T_h^\Omega;\mathbb R^2).
$$
Given $w_h^\Omega \in \mathcal S^1(\mathcal T_h^\Omega;\mathbb R^2)$, we say that $D w_h^\Omega$ is the discrete Hessian of some $u_h^\Omega\in \mathcal S^1_0(\mathcal T_h^\Omega)$ if it satisfies
\begin{equation*}
\int_{\Omega} w_h^\Omega\cdot v =\int_{\Omega}\nabla u_h^\Omega\cdot v \qquad\forall v\in \mathcal S^1(\mathcal T_h^\Omega;\mathbb R^2)
\end{equation*}
and we write $D^2_h u_h^\Omega := D w_h^\Omega$. The discrete Hessian $D^2_h u_h^\Omega$ is expected to be discontinuous across the element boundaries. In order to define a function that represents the
evaluation of the discretized approximate Hamiltonian $H_{\sigma,h_Y}$ at $D^2_h u_h^\Omega$, we define the continuous and piecewise affine function $\tilde H_{\sigma,h_Y}(D^2u_h^\Omega)$ by nodal
averaging of the piecewise constant function
$$
 x \mapsto  H_{\sigma,h_Y}(\operatorname{mid}(T))
 \quad\text{for } T\in\mathcal T^\Omega_h \text{ with }x\in T
$$
(defined a.e.\ in $\Omega$) where $\operatorname{mid}(T)$ denotes the barycenter of $T$. We then define the numerical approximation $u_h^\Omega = u_h^\Omega(h_\Omega,\sigma,h_Y)$ as a minimizer of the following least-squares functional
$$
u_h^\Omega\in
\underset{v_h^\Omega\in \mathcal S^1_0(\mathcal T_h^\Omega)}{\operatorname{arg\,min}}
\| v_h^\Omega + \tilde H_{\sigma,h_Y}(D^2_h v_h^\Omega) \|_{L^2(\Omega)}^2 .
$$
In our implementation, we computed the minimizer by using Matlab's built-in function \texttt{fmincon}, without prescribing any derivative information because we are not aware of any (semi)smoothness properties of the solution operator. We choose $\sigma=0.1$ and $h_Y=\sqrt{2}\times 2^{-2}$ fixed and consider a sequence of uniformly refined triangulations of $\Omega$ with mesh sizes $h_\Omega \in \sqrt{2}\times 2^{-\{1,2,3,4\}}$. For the error computation, we use as reference solution the approximation of $u_\varepsilon$ with $\varepsilon=0.1$ on a triangulation with mesh-size $\sqrt{2}\times2^{-7}$. The convergence history of the errors in the $L^\infty$ and $L^2$ norms is displayed in Figure~\ref{f:convergence_effective_problem}.

\begin{figure}[H]
\begin{tikzpicture}[scale=1.15]
  \begin{loglogaxis}
        [
            xlabel=$1/h_\Omega$,
            ylabel style={rotate=0},
            ylabel={relative error},
            xmin=1e0,xmax=1e2,
            ymin=1e-2,ymax=1e1,
            legend style=
            {
                at={(1,.5)},
                anchor=west,
                draw=black,
                fill=none
            },
            cycle list name=black white,
%             legend columns =2,
%             legend cell align=left,
            grid=major,
            clip=false,
            scaled ticks=true
        ]
      \addplot table[x index =0,y index=1]{./numresults/hjb_effective/relerrL2vshinv.dat};
      \addplot table[x index =0,y index=1]{./numresults/hjb_effective/relerrMAXvshinv.dat};
      \legend{
             $\frac{\|u_h^\Omega-u_\varepsilon\|_{L^\infty(\Omega)}}
                   {\|u_\varepsilon\|_{L^\infty(\Omega)}}$,
             $\frac{\|u_h^\Omega-u_\varepsilon\|_{L^2(\Omega)}}
                   {\|u_\varepsilon\|_{L^2(\Omega)}}$,
            }
    \end{loglogaxis}
  \end{tikzpicture}
\caption{Convergence history under mesh refinement of $\Omega$ for the approximation of the solution $u_0$ to the effective equation. The reference solution $u_\varepsilon$ is computed for $\varepsilon=0.1$. The cell problem is solved with $h_Y=\sqrt{2}\times 2^{-2}$ and $\sigma=0.1$.}
  \label{f:convergence_effective_problem}
\end{figure}
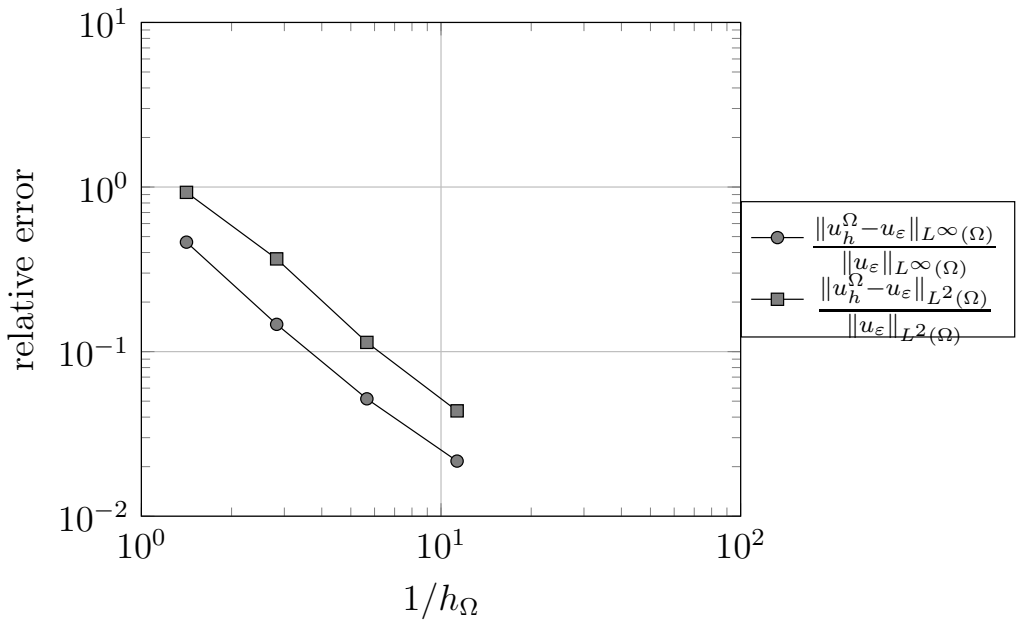

For both error norms we observe a convergence order of $h_\Omega^{3/2}$, which indicates that the effective problem with the chosen data is possibly more regular than predicted in Remark~\ref{r:regularity hom sol}.

\section{Collection of Proofs}\label{Sec Coll of Pfs}

\subsection{Proofs for Section \ref{Section FE schemes for periodic HJB}}

\begin{proof}[Proof of Lemma \ref{Preliminary estimates}]
For the first part, we use successively properties for the supremum, the Cauchy--Schwarz inequality, simple calculation and the Cordes condition \eqref{Cordes periodic} to obtain
\begin{align*}
&\left\lvert F_{\gamma}[(w,u)]-F_{\gamma}[(w',u')]-L_{\lambda}(w-w',u-u')\right\rvert^2 \\ &\leq \sup_{\alpha\in\Lambda} \left\lvert \gamma^{\alpha}\left(-A^{\alpha}:D(w-w')-b^{\alpha}\cdot \nabla (u-u')+c^{\alpha}(u-u')\right)+\nabla\cdot (w-w')-\lambda (u-u')\right\rvert^2\\ &\leq \sup_{\alpha\in\Lambda} \left( \left\lvert -\gamma^{\alpha}A^{\alpha}+I\right\rvert^2+\frac{\lvert \gamma^{\alpha}b^{\alpha}\rvert^2}{2\lambda}+\frac{\lvert \gamma^{\alpha}c^{\alpha}-\lambda\rvert^2 }{\lambda^2}\right)\left(\lvert D(w-w')\rvert^2 + 2\lambda\left\lvert \nabla (u-u')\right\rvert^2 +\lambda^2\lvert u-u'\rvert^2   \right)\\ &= \sup_{\alpha\in\Lambda} \left( n+1-\frac{\left(\mathrm{tr}(A^{\alpha})+\frac{c^{\alpha}}{\lambda}\right)^2}{\lvert A^{\alpha}\rvert^2+\frac{\lvert b^{\alpha}\rvert^2}{2\lambda}+\frac{\lvert c^{\alpha}\rvert^2}{\lambda^2}}\right)\left(\lvert D(w-w')\rvert^2 + 2\lambda\left\lvert \nabla (u-u')\right\rvert^2 +\lambda^2\lvert u-u'\rvert^2   \right)\\
&\leq (1-\delta)\left(\lvert D(w-w')\rvert^2 + 2\lambda\left\lvert \nabla (u-u')\right\rvert^2 +\lambda^2\lvert u-u'\rvert^2   \right)
\end{align*}
almost everywhere in $Y$, which yields the estimate \eqref{key Cordes implication}.

For the second part, we use \eqref{Maxwell-type}, integration by parts and Young's inequality to find
\begin{align*}
\vertiii{(w,u)}_{\lambda}^2 &= \|\mathrm{rot}(w)\|_{L^2(Y)}^2+\|\nabla\cdot w\|_{L^2(Y)}^2+2\lambda \|\nabla u\|_{L^2(Y)}^2+\lambda^2 \|u\|_{L^2(Y)}^2 \\
&= \|\mathrm{rot}(w)\|_{L^2(Y)}^2 + \|-\nabla\cdot w+\lambda u\|_{L^2(Y)}^2 + 2\lambda \int_Y \left(\nabla u - w\right)\cdot \nabla u\\
&\leq \|\mathrm{rot}(w)\|_{L^2(Y)}^2 + \|L_{\lambda}(w,u)\|_{L^2(Y)}^2 + \frac{\lambda}{\rho}\|\nabla u - w\|_{L^2(Y)}^2+ \lambda\rho \|\nabla u\|_{L^2(Y)}^2\\
&\leq \|\mathrm{rot}(w)\|_{L^2(Y)}^2 + \|L_{\lambda}(w,u)\|_{L^2(Y)}^2 + \frac{\lambda}{\rho}\|\nabla u - w\|_{L^2(Y)}^2+\frac{\rho }{2}\vertiii{(w,u)}_{\lambda}^2,
\end{align*}
which yields the Miranda--Talenti-type estimate \eqref{MT}.
\end{proof}

\begin{proof}[Proof of Lemma \ref{Mon and Lip}]
We are going to prove the claimed results (i), (ii), (iii) separately.

(i) By \eqref{key Cordes implication}, Young's inequality and the Miranda--Talenti-type estimate \eqref{MT} with the choice $\rho=2-2\sqrt{1-\delta}$, we find that
\begin{align*}
&a\left((w,u),(w-w',u-u')  \right)-a\left((w',u'),(w-w',u-u')  \right) \\ &\hspace{3cm} - \sigma_1 \|\mathrm{rot}(w-w')\|_{L^2(Y)}^2 -\sigma_2 \|\nabla(u-u')-(w-w')\|_{L^2(Y)}^2\\
&=\int_Y \left( F_{\gamma}[(w,u)]-F_{\gamma}[(w',u')]\right)L_{\lambda} (w-w',u-u')\\  &\geq \|L_{\lambda} (w-w',u-u')\|_{L^2(Y)}^2-\sqrt{1-\delta}\,\vertiii{(w-w',u-u')}_{\lambda}\|L_{\lambda} (w-w',u-u')\|_{L^2(Y)}\\
&\geq \frac{2-\sqrt{1-\delta}}{2}\|L_{\lambda} (w-w',u-u')\|_{L^2(Y)}^2-\frac{\sqrt{1-\delta}}{2}\vertiii{(w-w',u-u')}_{\lambda}^2 \\ &\geq \frac{1-\sqrt{1-\delta}}{2}\|L_{\lambda} (w-w',u-u')\|_{L^2(Y)}^2 - \frac{1}{2}\|\mathrm{rot}(w-w')\|_{L^2(Y)}^2 \\ &\hspace{3cm}-\frac{\lambda}{4-4\sqrt{1-\delta}}\|\nabla(u-u')-(w-w')\|_{L^2(Y)}^2.
\end{align*}
Therefore, by the definition of the constants $\sigma_1,\sigma_2$ and the Miranda--Talenti-type estimate \eqref{MT} with the choice $\rho=1$, we conclude that
\begin{align*}
a&\left((w,u),(w-w',u-u')  \right)-a\left((w',u'),(w-w',u-u')  \right) \\
&\geq \frac{1-\sqrt{1-\delta}}{2}\left(\|L_{\lambda} (w-w',u-u')\|_{L^2(Y)}^2+\|\mathrm{rot}(w-w')\|_{L^2(Y)}^2+\lambda \|\nabla(u-u')-(w-w')\|_{L^2(Y)}^2\right)\\ &\geq \frac{1-\sqrt{1-\delta}}{4}\vertiii{(w-w',u-u')}_{\lambda}^2,
\end{align*}
which is the claimed inequality.

(ii) We note that we have
\begin{align}\label{L_lambda est}
\|L_{\lambda}(w,u)\|_{L^2(Y)}\leq \sqrt{2} \,\vertiii{(w,u)}_{\lambda}\qquad \forall\, (w,u)\in X,
\end{align}
as there holds $\|\nabla\cdot w\|_{L^2(Y)}\leq \|Dw\|_{L^2(Y)}$ for any $w\in W_{\mathrm{per}}(Y;\R^n)$ by \eqref{Maxwell-type}. We bound the terms arising in the quantity on the left-hand side of \eqref{Lips} separately. For the term involving the nonlinearity, using \eqref{key Cordes implication}, we have
\begin{align*}
&\left\lvert \int_Y \left( F_{\gamma}[(w,u)]-F_{\gamma}[(w',u')]\right)L_{\lambda} (z,v)\right\rvert \\&\leq \|L_{\lambda} (z,v) \|_{L^2(Y)}\left( \|L_{\lambda} (w-w',u-u')\|_{L^2(Y)}+\sqrt{1-\delta}\,\vertiii{(w-w',u-u')}_{\lambda}\right)\\ &\leq \sqrt{2}\left(\sqrt{2}+\sqrt{1-\delta}\right)\vertiii{(w-w',u-u')}_{\lambda}\vertiii{(z,v)}_{\lambda}.
\end{align*}
For the term multiplying the constant $\sigma_1$, we have
\begin{align*}
\left\lvert \sigma_1\int_Y \mathrm{rot}(w-w')\cdot \mathrm{rot}(z)\right\rvert \leq \sigma_1 \|D(w-w')\|_{L^2(Y)}\|Dz\|_{L^2(Y)}\leq \sigma_1 \vertiii{(w-w',u-u')}_{\lambda}\vertiii{(z,v)}_{\lambda} ,
\end{align*}
as there holds $\|\mathrm{rot}(w)\|_{L^2(Y)}\leq \|Dw\|_{L^2(Y)}$ for any $w\in W_{\mathrm{per}}(Y;\R^n)$ by \eqref{Maxwell-type}. For the term multiplying the constant $\sigma_2$, we have by the triangle, Poincar\'e \eqref{Poincare explicit} and Cauchy--Schwarz inequalities that
\begin{align*}
&\left\lvert \sigma_2\int_Y \left(\nabla(u-u')-(w-w')\right)\cdot (\nabla v-z)\right\rvert \\ &\leq \sigma_2 \left( \|\nabla(u-u')\|_{L^2(Y)}+\frac{\sqrt{n}}{\pi}\|D(w-w')\|_{L^2(Y)}\right)\left(\|\nabla v\|_{L^2(Y)}+\frac{\sqrt{n}}{\pi}\|Dz\|_{L^2(Y)}\right)\\
&\leq \sigma_2 \left(\frac{1}{2\lambda}+\frac{n}{\pi^2}\right)\vertiii{(w-w',u-u')}_{\lambda}\vertiii{(z,v)}_{\lambda}.
\end{align*}
Altogether, we obtain the claimed inequality \eqref{Lips} with the constant
\begin{align*}
C_L=2+\sqrt{2}\sqrt{1-\delta}+\sigma_1 +\sigma_2\left(\frac{1}{2\lambda}+\frac{n}{\pi^2} \right),
\end{align*}
which is identical to the one given in Lemma \ref{Mon and Lip} (ii) using that $\sigma_2=\lambda\,\tilde{\sigma}_2$.

(iii) For any $m'\in M\backslash \{0\}$ we have $(0,m')\in X$ and hence
\begin{align*}
\sup_{(w',u')\in X\backslash\{0\}} \frac{b(m',(w',u'))}{\vertiii{(w',u')}_{\lambda}}\geq \frac{b(m',(0,m'))}{\vertiii{(0,m')}_{\lambda}}=\frac{\|\nabla m'\|_{L^2(Y)}^2}{\sqrt{2\lambda \|\nabla m'\|_{L^2(Y)}^2+\lambda^2 \|m'\|_{L^2(Y)}^2}}
\geq \frac{\|\nabla m'\|_{L^2(Y)}}{\sqrt{2\lambda+\frac{n}{\pi^2}\lambda^2}}
\end{align*}
by Poincar\'{e}'s inequality \eqref{Poincare explicit scalar} (recall that $M\subset W_{\mathrm{per}}(Y)$), which yields the claimed result \eqref{infsup}.
\end{proof}

\begin{proof}[Proof of Theorem \ref{wp of mixed}]
The existence of a unique solution $(m,(w,u))\in M\times X$ to \eqref{Mixed formulation} follows from the Brezzi-splitting; see \cite{BBF13} and \cite[Proposition 2.5]{GS19}, as we have the monotonicity and Lipschitz continuity for $a$ and an inf-sup condition from Lemma \ref{Mon and Lip}. For the second part of the claim, i.e., that $m=0$, $u\in H^2_{\mathrm{per}}(Y)$ with $w=\nabla u$ and $u$ is the solution to \eqref{HJB equ}, we note that $L_{\lambda}$ is surjective from the set $X_g:=\{(w',u')\in X:w'=\nabla u'\}$ onto $L^2(Y)$. We first test the mixed formulation \eqref{Mixed formulation} with pairs $(w',u')$ from $X_g$ to obtain $F_{\gamma}[(w,u)]=0$ almost everywhere and then with the solution pair $(w,u)$ to find that $w=\nabla u$ and thus $u\in H^2_{\mathrm{per}}(Y)$. We conclude the proof by noting that this implies that $u$ is the solution to \eqref{HJB modified} (and hence to \eqref{HJB equ} by Theorem \ref{thm well-pos}) and that $m=0$.
\end{proof}

\begin{proof}[Proof of Lemma \ref{lmm bddness and disinfsup}]
We use the triangle, Poincar\'e \eqref{Poincare explicit} and Cauchy--Schwarz inequalities to obtain that
\begin{align*}
b(m',(w',u'))\leq \|\nabla m'\|_{L^2(Y)}\left(\|\nabla u'\|_{L^2(Y)}+\frac{\sqrt{n}}{\pi}\|Dw'\|_{L^2(Y)}\right)\leq \sqrt{\frac{1}{2\lambda}+\frac{n}{\pi^2}}\, \|\nabla m'\|_{L^2(Y)}\vertiii{(w',u')}_{\lambda}
\end{align*}
for all $(m',(w',u'))\in M\times X$.

The discrete inf-sup condition holds, as for $m'_h\in M_h\backslash \{0\}$ we have $(0,m'_h)\in X_h\backslash\{0\}$ since $M_h\subset U_h\cap M$, and hence
\begin{align*}
\sup_{(w'_h,u'_h)\in X_h\backslash\{0\}} \frac{b(m'_h,(w'_h,u'_h))}{\vertiii{(w'_h,u'_h)}_{\lambda}}\geq \frac{b(m'_h,(0,m'_h))}{\vertiii{(0,m'_h)}_{\lambda}}=\frac{\|\nabla m'_h\|_{L^2(Y)}^2}{\sqrt{2\lambda \|\nabla m'_h\|_{L^2(Y)}^2+\lambda^2 \|m'_h\|_{L^2(Y)}^2}}
\geq \frac{\|\nabla m'_h\|_{L^2(Y)}}{\sqrt{2\lambda+\frac{n}{\pi^2}\lambda^2}}
\end{align*}
by Poincar\'{e}'s inequality \eqref{Poincare explicit scalar} (recall that $M\subset W_{\mathrm{per}}(Y)$), which yields the claimed result \eqref{dis infsup}.
\end{proof}

\begin{proof}[Proof of Theorem \ref{Thm error bd}]
We only show the error bound, as the existence and uniqueness of solutions for \eqref{Discrete mixed formulation} follows from Lemma \ref{Mon and Lip} and Lemma \ref{lmm bddness and disinfsup} in a standard way; see \cite[Proposition 3.1]{GS19}.

\textit{Step 1}: We introduce the discrete kernel
\begin{align*}
Z_h:=\left\{ (w_h',u_h')\in X_h: b\left(m_h',(w_h',u_h') \right)=0\quad\forall\, m_h'\in M_h\right\}
\end{align*}
and claim that there holds
\begin{align}\label{Goal of Step 1}
\vertiii{(w-w_h,u-u_h)}_{\lambda}\leq  \frac{C_L}{C_M} \inf_{(w'_h,u'_h)\in Z_h}\vertiii{(w-w'_h,u-u'_h)}_{\lambda}.
\end{align}
Indeed, we use successively the monotonicity from Lemma \ref{Mon and Lip} (i), the solution property of $(w,u)$ from Theorem \ref{wp of mixed} and the fact that $(w_h,u_h)$ solves the discrete problem \eqref{Discrete mixed formulation}, and the Lipschitz estimate from  Lemma \ref{Mon and Lip} (ii) to find that
\begin{align*}
C_M\vertiii{(w-w_h,u-u_h)}_{\lambda}^2&\leq a\left((w,u),(w-w_h,u-u_h)  \right)-a\left((w_h,u_h),(w-w_h,u-u_h)  \right)\\
&=-a\left((w_h,u_h),(w,u)\right)\\
&=-a\left((w_h,u_h),(w-w_h',u-u_h')  \right)\\
&=a\left((w,u),(w-w_h',u-u_h')\right)-a\left((w_h,u_h),(w-w_h',u-u_h')  \right)\\
&\leq C_L\vertiii{(w-w_h,u-u_h)}_{\lambda}\vertiii{(w-w_h',u-u_h')}_{\lambda}
\end{align*}
for any $(w'_h,u'_h)\in Z_h$, which implies the desired estimate \eqref{Goal of Step 1}.

\textit{Step 2}: We let $(w_{*},u_{*})\in X_h$ denote the best-approximation to $(w,u)$ from $Z_h$, i.e.,
\begin{align}\label{defn of wstar}
\vertiii{(w-w_{*},u-u_{*})}_{\lambda}= \inf_{(w'_h,u'_h)\in Z_h}\vertiii{(w-w'_h,u-u'_h)}_{\lambda},
\end{align}
and we derive a linear mixed problem for $(w_{*},u_{*})$.

By the discrete inf-sup condition \eqref{dis infsup}, there exists $m_{*}\in M_h$ such that
\begin{align*}
\left\{ \begin{aligned} \left\langle(w_{*},u_{*}),(w'_h,u'_h) \right\rangle_{\lambda}+ b(m_{*},(w'_h,u'_h)) &= \left\langle(w,u),(w'_h,u'_h) \right\rangle_{\lambda} & & \forall \,(w'_h,u'_h)\in X_h,\\ b(m'_h,(w_{*},u_{*}))&=0 & &\forall\, m'_h\in M_h,\end{aligned}\right.
\end{align*}
where $\langle\cdot,\cdot\rangle_{\lambda}:X\times X\rightarrow \R$ is the inner product given by
\begin{align*}
\langle (w',u'), (w'',u'')\rangle_{\lambda}:= \int_Y Dw':Dw'' + 2\lambda\int_Y \nabla u'\cdot \nabla u'' + \lambda^2\int_Y u'u''.
\end{align*}
We also note that the solution pair $(m,(w,u))$ satisfies the similar system (recall that $m=0$)
\begin{align*}
\left\{ \begin{aligned} \left\langle(w,u),(w',u') \right\rangle_{\lambda}+ b(m,(w',u')) &= \left\langle(w,u),(w',u') \right\rangle_{\lambda} & & \forall \,(w',u')\in X,\\ b(m',(w,u))&=0 & &\forall\, m'\in M.\end{aligned}\right.
\end{align*}

\textit{Step 3}: We derive an error bound for $(w-w_{*},u-u_{*})$ in the $\vertiii{\cdot}_{\lambda}$ norm using classical linear mixed finite element theory.

Note that for any $(w',u'),(w'',u'')\in X$, we have
\begin{align*}
\left\lvert\langle (w',u'), (w'',u'')\rangle_{\lambda}\right\rvert\leq \vertiii{(w',u')}_{\lambda} \vertiii{(w'',u'')}_{\lambda},\qquad \langle (w',u'), (w',u')\rangle_{\lambda} = \vertiii{(w',u')}_{\lambda}^2.
\end{align*}
In particular, we have boundedness and coercivity on the whole space, i.e.,
\begin{align*}
\left\lvert\langle (w',u'), (w'',u'')\rangle_{\lambda}\right\rvert\leq C_a \vertiii{(w',u')}_{\lambda} \vertiii{(w'',u'')}_{\lambda},\qquad \langle (w',u'), (w',u')\rangle_{\lambda} \geq c_a \vertiii{(w',u')}_{\lambda}^2
\end{align*}
for all $(w',u'),(w'',u'')\in X$ with the constants $C_a:=c_a:=1$. Further, from Lemma \ref{lmm bddness and disinfsup} we have the discrete inf-sup condition \eqref{dis infsup} with the constant $c_b$ and boundedness of $b$ with the constant $C_b$. Then, by linear mixed finite element theory (see \cite{Sul13}), we obtain
\begin{align}\label{Result of Step 3}
\begin{split}
&\vertiii{(w-w_{*},u-u_{*})}_{\lambda} \\&\leq \left(1+\frac{C_a}{c_a}\right)\left(1+\frac{C_b}{c_b}\right)\inf_{(w'_h,u'_h)\in X_h}\vertiii{(w-w_h',u-u_h')}_{\lambda}+\frac{C_b}{c_a}\inf_{m'_h\in M_h} \|\nabla(m-m_h')\|_{L^2(Y)}\\
&=2\left(1+\frac{C_b}{c_b}\right)\inf_{(w'_h,u'_h)\in X_h}\vertiii{(w-w_h',u-u_h')}_{\lambda},
\end{split}
\end{align}
where we have used that $m=0$ and $C_a=c_a=1$ in the last line.

\textit{Step 4}: We conclude by combining \eqref{Goal of Step 1}, \eqref{defn of wstar} and \eqref{Result of Step 3} to deduce that
\begin{align*}
\vertiii{(w-w_h,u-u_h)}_{\lambda}\leq \frac{C_L}{C_M} \vertiii{(w-w_{*},u-u_{*})}_{\lambda} \leq
2\frac{C_L}{C_M}\left(1+\frac{C_b}{c_b}\right)\inf_{(w'_h,u'_h)\in X_h}\vertiii{(w-w_h',u-u_h')}_{\lambda},
\end{align*}
which is the desired error bound.
\end{proof}

\begin{proof}[Proof of Theorem \ref{thrm a posteriori}]
We use successively the monotonicity from Lemma \ref{Mon and Lip} (i), the solution property of $(w,u)$ from Theorem \ref{wp of mixed}, the Cauchy--Schwarz inequality (note that $w=\nabla u$), the bound \eqref{L_lambda est}, and Young's inequality to show that
\begin{align*}
&C_M\vertiii{(w-w_h,u-u_h)}_{\lambda}^2 \\&\leq a\left((w,u),(w-w_h,u-u_h)\right)-a\left((w_h,u_h),(w-w_h,u-u_h)\right) \\
&= -a\left((w_h,u_h),(w-w_h,u-u_h)\right)\\
&\leq \left\|F_{\gamma}[(w_h,u_h)]\right\|_{L^2(Y)}\left\|L_{\lambda}(w-w_h,u-u_h) \right\|_{L^2(Y)}+\sigma_1 \left\|\mathrm{rot}(w_h)\right\|_{L^2(Y)}^2+\sigma_2 \left\|w_h-\nabla u_h\right\|_{L^2(Y)}^2\\
&\leq \sqrt{2}  \left\|F_{\gamma}[(w_h,u_h)]\right\|_{L^2(Y)}\vertiii{(w-w_h,u-u_h)}_{\lambda}+\sigma_1 \left\|\mathrm{rot}(w_h)\right\|_{L^2(Y)}^2+\sigma_2 \left\|w_h-\nabla u_h\right\|_{L^2(Y)}^2\\
&\leq C_M^{-1}\left\|F_{\gamma}[(w_h,u_h)]\right\|_{L^2(Y)}^2+\frac{C_M}{2}\vertiii{(w-w_h,u-u_h)}_{\lambda}^2+\sigma_1 \left\|\mathrm{rot}(w_h)\right\|_{L^2(Y)}^2+\sigma_2 \left\|w_h-\nabla u_h\right\|_{L^2(Y)}^2.
\end{align*}
Upon rearranging, we find the claimed \textit{a posteriori} estimate.

For the efficiency estimate, recall the solution property of $(w,u)$ from Theorem \ref{wp of mixed}, in particular $w=\nabla u$ and $F_{\gamma}[(w,u)]=0$ almost everywhere.
With the Lipschitz property from Lemma \ref{Mon and Lip} (ii) and with Lemma \ref{Preliminary estimates}, we then obtain
\begin{align*}
&C_L \vertiii{(w-w_h,u-u_h)}_{\lambda}^2 - \sigma_1 \left\|\mathrm{rot}(w_h)\right\|_{L^2(Y)}^2-\sigma_2 \left\|w_h-\nabla u_h\right\|_{L^2(Y)}^2 \\ &\geq a\left((w,u),(w-w_h,u-u_h)\right)-a\left((w_h,u_h),(w-w_h,u-u_h)\right)- \sigma_1 \left\|\mathrm{rot}(w_h)\right\|_{L^2}^2-\sigma_2 \left\|w_h-\nabla u_h\right\|_{L^2}^2\\
&=-a\left((w_h,u_h),(w-w_h,u-u_h)\right)- \sigma_1 \left\|\mathrm{rot}(w_h)\right\|_{L^2(Y)}^2-\sigma_2 \left\|w_h-\nabla u_h\right\|_{L^2(Y)}^2\\
&=\left\|F_{\gamma}[(w_h,u_h)]\right\|_{L^2(Y)}^2
+ \int_Y F_{\gamma}[(w_h,u_h)]\left( F_{\gamma}[(w,u)]-F_{\gamma}[(w_h,u_h)]-L_{\lambda}(w-w_h,u-u_h)\right)\\
&\geq \left\|F_{\gamma}[(w_h,u_h)]\right\|_{L^2(Y)}^2-\sqrt{1-\delta} \left\|F_{\gamma}[(w_h,u_h)]\right\|_{L^2(Y)}\vertiii{(w-w_h,u-u_h)}_{\lambda}\\
&\geq \frac{1}{2}\left\|F_{\gamma}[(w_h,u_h)]\right\|_{L^2(Y)}^2-\frac{1-\delta}{2}\vertiii{(w-w_h,u-u_h)}_{\lambda}^2,
\end{align*}
which yields the efficiency estimate upon rearranging.
\end{proof}

\subsection{Proofs for Section \ref{Section NumHom}}

\begin{proof}[Proof of Theorem \ref{thrm error bd mixed}]
Using the definition of the $\vertiii{\cdot}_{\lambda_{\sigma}}$ norm and interpolation inequalities, denoting the interpolation operators on the finite element spaces by $\mathcal{I}^{\calS^q}_h,\mathcal{I}^{\calS^l}_h$, we find that
\begin{align*}
\inf_{(w'_h,u'_h)\in X_h}&\vertiii{(\nabla v^{\sigma}-w'_h, v^{\sigma}-u'_h)}_{\lambda_{\sigma}}\leq \vertiii{\left(\nabla v^{\sigma}-\left(\mathcal{I}^{\calS^q}_h(\nabla v^{\sigma})-\int_Y \mathcal{I}^{\calS^q}_h(\nabla v^{\sigma})\right) , v^{\sigma}-\mathcal{I}^{\calS^l}_h(v^{\sigma})\right)}_{\lambda_{\sigma}}
\\&=\left(\left\| D\left(\nabla v^{\sigma}-\mathcal{I}^{\calS^q}_h(\nabla v^{\sigma}) \right)\right\|_{L^2(Y)}^2+ 2\lambda_{\sigma} \lvert v^{\sigma}-\mathcal{I}^{\calS^l}_h(v^{\sigma})\rvert_{H^1(Y)}^2+\lambda_{\sigma}^2 \|v^{\sigma}-\mathcal{I}^{\calS^l}_h(v^{\sigma})\|_{L^2(Y)}^2 \right)^{\frac{1}{2}}\\&\leq C_i \left(h^{2\min\{r,q\}}  +2\lambda_{\sigma}h^{2\min\{1+r,l\}} +  \lambda_{\sigma}^2 h^{2\min\{2+r,l\}}   \right)^{\frac{1}{2}}\|\nabla v^{\sigma}\|_{H^{1+r}(Y)}\\ &\leq C_i \left(1  +2\lambda_{\sigma} +  \lambda_{\sigma}^2 \right)^{\frac{1}{2}}h^{\min\{r,q,l\}} \|\nabla v^{\sigma}\|_{H^{1+r}(Y)}
\end{align*}
for $h>0$ sufficiently small, where $C_i>0$ is the constant arising in applying the interpolation inequalities. The claimed result now follows from \eqref{appr mix}, i.e.,
\begin{align*}
\vertiii{(\nabla v^{\sigma}-w_h^{\sigma},v^{\sigma}-v_h^{\sigma})}_{\lambda_{\sigma}}&\leq C_e(\delta,\lambda_{\sigma},n) \inf_{(w'_h,u'_h)\in X_h}\vertiii{(\nabla v^{\sigma}-w'_h, v^{\sigma}-u'_h)}_{\lambda_{\sigma}}\\ &\leq
C_e(\delta,\lambda_{\sigma},n)C_i\left(1  +\lambda_{\sigma} \right)h^{\min\{r,q,l\}} \|\nabla v^{\sigma}\|_{H^{1+r}(Y)}\\
&\leq C_e(\delta,\lambda,n)C_i\left(1  +\lambda\right)h^{\min\{r,q,l\}} \|\nabla v^{\sigma}\|_{H^{1+r}(Y)},
\end{align*}
where we used $\lambda_{\sigma}\leq \lambda$ and Remark \ref{rk explicit Ce}.
\end{proof}

\begin{proof}[Proof of Theorem \ref{App of eff H}]
We use H\"{o}lder and triangle inequalities, Lemma \ref{Lmm: prop of H} and the error bound \eqref{helperror} to obtain
\begin{align*}
\left\lvert\int_Y \left(-\sigma v^{\sigma}_{h}(\cdot\,;s,p,R)\right) - H(s,p,R)\right\rvert &= \left\lvert\int_Y \left(-\sigma v^{\sigma}_{h}(\cdot\,;s,p,R) - H(s,p,R)\right)\right\rvert\\
&\leq \|-\sigma v^{\sigma}_{h}(\cdot\,;s,p,R) - H(s,p,R)\|_{L^2(Y)}\\
&\lesssim \sigma \|v^{\sigma}_{h}(\cdot\,;s,p,R)-v^{\sigma}(\cdot\,;s,p,R)\|_{L^2(Y)}+ \sigma(1+\lvert p\rvert +\lvert R\rvert)\\
&\lesssim \left(h^{\min\{r,q,l\}} + \sigma\right)(1+\lvert p\rvert +\lvert R\rvert).
\end{align*}
The second part of the claim can be shown analogously.
\end{proof}

\section{Conclusion}

In this work we introduced a scheme for the numerical homogenization of the fully nonlinear second-order Hamilton--Jacobi--Bellman equation with Cordes coefficients, based on a novel mixed finite element method for the periodic corrector problems.

The focus of the first part of the paper was the construction and the rigorous analysis of mixed finite element approximations to the periodic solution of the HJB equation. We derived a mixed formulation for the problem and proved well-posedness as well as \textit{a priori} and \textit{a posteriori} error bounds. Explicit formulas for the error constants were provided, showing the asymptotic behavior of the constants in the Cordes parameters.

In the second part of the paper we focused on the numerical homogenization of HJB equations with locally periodic coefficients. Theoretical homogenization results were provided and used in the analysis of the numerical homogenization scheme. We presented and rigorously analyzed a method for the approximation of the effective Hamiltonian based on mixed finite element approximations of the periodic cell problem for the approximate corrector from the first part.

Finally, we presented numerical experiments illustrating the theoretical results. The experiments demonstrated the approximation of the effective Hamiltonian at a point as well as the approximation of the solution to the homogenized problem. 

Future work will focus on the numerical homogenization of other fully nonlinear partial differential equations such as the Isaacs equation. The strong $H^2$ solution of Isaacs equations with Cordes coefficients has recently been discussed in \cite{KS20} and can be used as a framework to study its numerical homogenization.

\addtocontents{toc}{\protect\setcounter{tocdepth}{0}}

\section*{Acknowledgements}

This work was supported by the UK Engineering and Physical Sciences Research Council [EP/L015811/1].
The second author gratefully acknowledges helpful conversations with Professor Yves Capdeboscq (Universit\'{e} de Paris) during the preparation of this work.

\bibliographystyle{plain}
\bibliography{ref}

\end{document}